\documentclass[11pt]{article}
\usepackage[utf8]{inputenc}
\usepackage[a4paper, margin=1in]{geometry}
\usepackage{amsmath, amssymb, amsthm, mathtools}
\usepackage{graphicx, caption, subcaption}
\usepackage{enumitem}
\usepackage{hyperref}
\usepackage{color}
\usepackage{float}
\usepackage{xcolor}
\usepackage{natbib}

\usepackage{comment}
\usepackage{times}

\usepackage{setspace}
\usepackage{changepage}

\renewenvironment{abstract}{%
	\begin{center}%
		\normalfont\fontsize{10pt}{12pt}\selectfont
		\begin{adjustwidth}{1.5pc}{1.5pc} % indent left & right
		}{%
		\end{adjustwidth}%
	\end{center}%
}

\title{\textbf{Stability and Bifurcation Analysis of Fractional Delay Differential Equation with a Delay-dependent Coefficient}}
\author{  Pragati Dutta, Sachin Bhalekar\footnote{corresponding author}\\
	\textit{School of Mathematics and Statistics, University of Hyderabad,India}\\
	email ids: 23mmpa01@uohyd.ac.in (PD)
	sachinbhalekar@uohyd.ac.in (SB)}
\date{}

\newtheorem{theorem}{Theorem}[section]

\theoremstyle{remark}

\newtheorem{definition}{Definition}[section]

\begin{document}
	\maketitle
	
	\begin{abstract}
	This paper investigates the stability of different regions in the $(k,\gamma)$-plane for a class of fractional delay differential equations given by
	\begin{equation}
		D^{\alpha} x(t)
		= -\gamma x(t)
		+ g\big(x(t - \tau_1)\big)
		- e^{-\gamma \tau_2}\, g\big(x(t - \tau_1 - \tau_2)\big),
		\qquad 0 < \alpha \le 1,
	\end{equation}
	where $k = g'(0)$. The primary focus is on the stability of the trivial equilibrium of the corresponding linearized system. A detailed stability and bifurcation analysis is carried out for the particular case $\tau_1 = 0$ and $\tau_2 \ge 0$. Furthermore, a general result is established for the case $\tau_1 > 0$, $\tau_2 \ge 0$, which holds for all values of $\alpha$ and $\tau_1$.\\
	In addition, illustrative examples are provided in the form of stability diagrams in the $(\tau_1,\tau_2)$-plane for fixed values of $\alpha$, $k$, and $\gamma$. These diagrams are generated using appropriate numerical methods to visualize the stability regions and to support the theoretical results.
		\end{abstract}
	

\section{Introduction}

In many real-world systems, the present state depends not only on current conditions but also on past behavior. Fractional calculus provides a natural framework to capture such memory effects through derivatives of non-integer order. Unlike classical integer-order derivatives, fractional derivatives are non-local operators and inherently account for system history \cite{podlubny1998fractional,kilbas2006theory}. This makes fractional differential equations (FDEs) particularly suitable for modeling phenomena in viscoelasticity, control systems, signal processing, and biological processes \cite{diethelm2010analysis,oldham1974j,mainardi2022fractional}. Moreover, the flexibility in the order of differentiation allows better agreement with experimental observations \cite{bagley1983theoretical,Podlubny1999}. The qualitative behavior and stability of such systems have been widely studied using Lyapunov-type and related analytical methods \cite{agarwal2016stability}.

Another important feature in modeling real systems is the presence of time delays, where the rate of change depends on past states. This leads to delay differential equations (DDEs), which arise naturally in population dynamics, physiology, engineering, and control theory \cite{hale2013introduction,erneux2009applied,kuang1993delay}. Although similar in form to ordinary differential equations, DDEs are infinite-dimensional and can exhibit oscillations, bifurcations, and chaos even in simple settings \cite{hale2013introduction,sprott2007simple,ruiz2013chaos,insperger2011semi}. They have been widely applied across scientific and engineering disciplines \cite{michiels2010control,gumussoy2014computer,agarwal2016stability}. In this context, discrete and difference equation techniques also play an important role in the analysis of delay systems \cite{agarwal2000difference}.

The combination of fractional dynamics with time delays leads to fractional delay differential equations (FDDEs), which provide a more realistic framework for systems exhibiting both memory and delayed responses. Various numerical and analytical methods have been developed for such systems \cite{bhalekar2011predictor,daftardar2014new}, and their stability and bifurcation properties have been extensively studied \cite{bhalekar2016stability,bhalekar2011fractional,gupta2024fractional,bhalekar2022stability}. In particular, practical stability concepts and Razumikhin-type methods have been successfully applied to fractional delay systems \cite{agarwal2021practical,agarwal2022generalized}. Moreover, FDDEs are known to exhibit complex dynamics, including chaotic behavior \cite{bhalekar2012dynamical,daftardar2012dynamics}.

In many applications, systems involve more than one delay, each representing a different mechanism acting over its own time scale \cite{gu2003stability,niculescu2002delay}. Such multi-delay models appear in population dynamics and biological systems, for example in describing maturation and feedback processes \cite{braddock1983two,piotrowska2008hopf,boullu2020stability}. However, the presence of multiple delays significantly increases analytical complexity. In addition, delay-dependent coefficients further enhance modeling flexibility by allowing the influence of past states to vary explicitly with the delay. Models with such features have been analyzed using geometric and characteristic equation approaches \cite{an2019geometric}, and delay-dependent stability criteria have been developed for systems with fixed, distributed, and nonlinear delays \cite{huang2004analysis,fei2017delay}.

In our earlier work \cite{dutta2025some}, we investigated the stability of a nonlinear fractional delay differential equation with two discrete delays and obtained several delay-independent results in the $(k,\gamma)$-plane. However, a complete characterization of the remaining parameter regions and the combined effect of both delays was not fully addressed. Motivated by this, the present paper completes the analysis for the case $\tau_1 = 0$ and derives new results for the general two-delay system. We also study the interaction between the delays through stability diagrams in the $(\tau_1,\tau_2)$-plane, providing further insight into the system dynamics.

The rest of the paper is organized as follows. Section~\ref{sec:prelim} presents the necessary preliminaries. Section~\ref{sec:model} introduces the model and derives the characteristic equation. Section~\ref{sec:C1} analyzes the case $\tau_1 = 0$, while Section~\ref{sec:C2} provides a sufficient condition for instability in the general case. Section~\ref{sec:C3} presents illustrative examples with stability diagrams in the $(\tau_1,\tau_2)$-plane. Finally, Section~\ref{sec:conc} concludes the paper.

	\section{Preliminaries}
\label{sec:prelim}
In this section, we present some definitions \cite{podlubny1998fractional,kilbas2006theory,lakshmanan2011dynamics,smith2010introduction} and a result \cite{bhalekar2016stability} available in the literature.

\begin{definition}[Riemann--Liouville Fractional Integral]
	Let $f \in L^{1}(0,b)$ and $\mu > 0$. The Riemann--Liouville fractional integral of order $\mu$ is defined by
	\begin{equation*}
		I^{\mu} f(t)
		= \frac{1}{\Gamma(\mu)} \int_{0}^{t} (t-\tau)^{\mu-1} f(\tau)\, d\tau,
		\qquad 0 < t < b. 
	\end{equation*}
\end{definition}
\begin{definition}[Caputo Fractional Derivative]
	Let $f \in L^{1}(0,b)$, $\mu > 0$, and $m \in \mathbb{N}$ such that $m-1 < \mu \le m$. The Caputo fractional derivative of order $\mu$ is defined by
	\[
	D^{\mu} f(t) =
	\begin{cases}
		\dfrac{d^{m}}{dt^{m}} f(t), & \text{if } \mu = m, \\[8pt]
		I^{\,m-\mu} \left( \dfrac{d^{m}}{dt^{m}} f(t) \right), & \text{if } m-1 < \mu < m,
	\end{cases}
	\]
	where $I^{\mu}$ is the Riemann--Liouville fractional integral.
	
	Moreover, for $m-1 < \mu \le m$, we have
	\[
	I^{\mu} \, D^{\mu} f(t)
	= f(t) - \sum_{k=0}^{m-1} \frac{f^{(k)}(0)}{k!} t^{k}.
	\]
\end{definition}

\begin{definition}[Equilibrium Point]
	Consider the fractional delay differential equation
	\begin{equation*}
		D^{\alpha} x(t) = f\big(x(t), x(t-\tau)\big), \qquad 0 < \alpha \le 1,   \tag{*}
		\label{star}
	\end{equation*}
	where $\tau > 0$ and $f \in C^{1}(E)$ with $E \subset \mathbb{R}^{2}$ open. 
	
	A constant solution $x(t) = x^{*}$ pf (\ref{star}) is called an equilibrium point, it satisfies
	\[
	f(x^{*}, x^{*}) = 0.
	\]
\end{definition}

\begin{definition}[Initial Value Problem]
	Consider the delay differential equation together with the initial function
	\begin{equation*}
		x(t) = \phi(t), \qquad -\tau \le t \le 0,
	\end{equation*}
	where $\phi : [-\tau,0] \to \mathbb{R}$ is a continuous function.
	
	The corresponding solution is denoted by $x(t,\phi)$, and the norm of $\phi$ is defined as
	\[
	\|\phi\| = \sup_{-\tau \le t \le 0} |\phi(t)|.
	\]
\end{definition}

\begin{definition}[Stability]
	An equilibrium point $x^{*}$ is said to be stable if, for every $\varepsilon > 0$, there exists $\delta > 0$ such that
	\[
	\|\phi - x^{*}\| < \delta \;\Rightarrow\; |x(t,\phi) - x^{*}| < \varepsilon, \quad \text{for all } t \ge 0.
	\]
\end{definition}

\begin{definition}[Asymptotic Stability]
	An equilibrium point $x^{*}$ is said to be asymptotically stable if it is stable and there exists $b_{0} > 0$ such that
	\[
	\|\phi - x^{*}\| < b_{0} \;\Rightarrow\; \lim_{t \to \infty} x(t,\phi) = x^{*}.
	\]
\end{definition}

\begin{definition}[Instability]
	An equilibrium point that is not stable is called unstable.
\end{definition}
\begin{theorem}
	\label{thm:prelim}
	Consider the scalar FDE with a single discrete delay:
	\begin{equation}
		\label{eq1}
		D^\alpha x(t) = a x(t) + b x(t - \tau), \quad 0 < \alpha \le 1, \tau\ge0.
	\end{equation}
	Then, the zero equilibrium $x_*=0$ has the following stability behavior:
	\begin{enumerate}
		\item If $b \in (-\infty, -|a|)$, then the equilibrium is asymptotically stable for $\tau \in [0, \tau_{cr})$ and the system undergoes Hopf bifurcation at
		\[
		\tau_{*}' = \frac{\arccos\left( \frac{(a\cos(\frac{\alpha \pi}{2}) + \sqrt{b^2 - a^2 \sin^2(\frac{\alpha \pi}{2})})\cos(\frac{\alpha \pi}{2}) - a}{b} \right)}{\left( a \cos\left(\frac{\alpha \pi}{2}\right) + \sqrt{b^2 - a^2 \sin^2\left(\frac{\alpha \pi}{2}\right)} \right)^{1/\alpha}}.
		\label{taustar}
		\]
		\item If $b \in (-a, \infty)$, then the equilibrium is unstable for all $\tau \ge 0$.
		\item If $a < 0$ and $b \in (a, -a)$, then the equilibrium is asymptotically stable for all $\tau \ge 0$.
	\end{enumerate}
	\begin{figure}[H]
		\centering
		\includegraphics[scale=0.57]{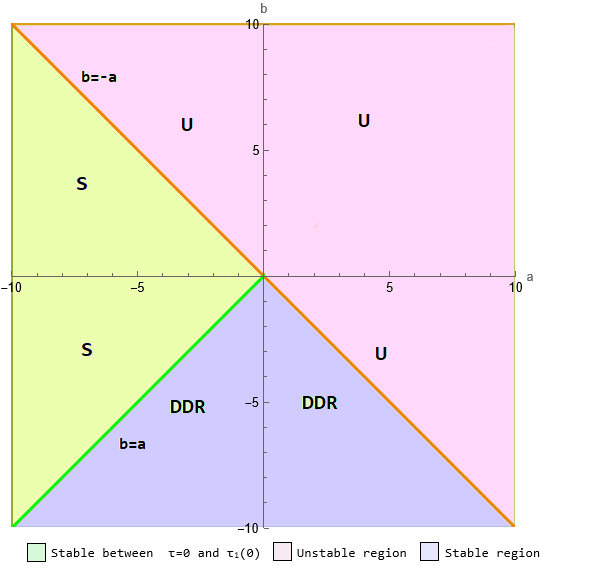}
		\caption{Stability regions of Equation (\ref{eq1}).}
		\label{fig:prel}
	\end{figure}
	\textnormal {Notation: Here, DDR means Delay Dependent Region.}
\end{theorem}

	\section{Model Formulation and Linearization}
	\label{sec:model}
	We consider the nonlinear fractional delay differential equation of the form
	\begin{equation}
		D^{\alpha} x(t)
		= -\gamma x(t)
		+ g\big(x(t - \tau_1)\big)
		- e^{-\gamma \tau_2}\, g\big(x(t - \tau_1 - \tau_2)\big),
		\qquad 0 < \alpha \le 1,
		\label{eq:main}
	\end{equation}
	where \( D^{\alpha} \) denotes the Caputo fractional derivative, \( \gamma \in \mathbb{R} \), and \( g \in C^{1}(\mathbb{R}) \).
	The parameters \( \tau_1, \tau_2 \ge 0 \) represent discrete delays, while the exponential factor \( e^{-\gamma \tau_2} \) introduces a delay-dependent coefficient.
	The integer-order counterpart of \eqref{eq:main}, analyzed in \cite{boullu2020stability}, arises in models describing platelet production dynamics.

	Let \( x_* \) denote an equilibrium point of \eqref{eq:main}. It satisfies
	\[
	-\gamma x_* + g(x_*) - e^{-\gamma \tau_2}\, g(x_*) = 0.
	\]
	Here, we assume $g$ to be a nonlinear function such that \( g(0) = 0 \), it immediately follows that \( x_* = 0 \) is an equilibrium of \eqref{eq:main}.
	
	Assuming further that \( g'(0) = k \), we linearize the system in a neighborhood of the equilibrium \( x_* = 0 \).
	
	The linearized form of the equation is given by
	\begin{equation}
		D^{\alpha} x(t)
		= -\gamma x(t)
		+ k\,x(t - \tau_1)
		- k e^{-\gamma \tau_2}\, x(t - \tau_1 - \tau_2).
		\label{eq:linear}
	\end{equation}
	The local stability of the nonlinear equation (\ref{eq:main}) in a neighborhood of the equilibrium $x_*=0$ coincides with that of its linearized equation.
	The corresponding characteristic equation associated with Eq.~\eqref{eq:linear} is
	\begin{equation}
		\lambda^{\alpha}
		= -\gamma
		+ k e^{-\lambda \tau_1}
		- k e^{-\gamma \tau_2} e^{-\lambda(\tau_1 + \tau_2)}.
		\label{eq:char}
	\end{equation}
	
	Some stability results for Eq.(\ref{eq:linear}) are presented in \cite{dutta2025some}. Most of those results are independent of delay. In this work, we continue that analysis and find few more stability results. We focus mainly on the delay-dependent stability. However, few results (Theorem 4.1) are independent of delay too. So, we discuss the delay-dependent regions for the following two cases: 
	\begin{itemize}
		\item \textbf{Case 1:} \( \tau_1 = 0 \) with arbitrary \( \tau_2 > 0 \),
		\item \textbf{Case 2:} \( \tau_1 > 0 \) with arbitrary \( \tau_2 > 0 \).
	\end{itemize}
	We also discuss stability diagrams in the \( \tau_1 \tau_2 \)-plane for selected parameter values.
	
	\section{Case 1: \texorpdfstring{$\tau_1 = 0$}{tau1 = 0}}
	\label{sec:C1}
	%In this case, the delay appears only in the exponentially modulated feedback term.% 
	Substituting $\tau_1 = 0$, $\tau_2 = \tau$ in the linearized equation (\ref{eq:linear}), we get
	\begin{equation}
		D^\alpha x(t) = (k - \gamma) x(t) - k e^{-\gamma \tau} x(t - \tau), \quad 0 < \alpha \le 1,
		\label{eq:case1}
	\end{equation}
	which can be written in the standard linear form
	\begin{equation}
		D^\alpha x(t) = a x(t) + b(\tau) x(t - \tau),
		\label{c1:lin}
	\end{equation}
	where
	\[
	a = k - \gamma, \qquad b(\tau) = -k e^{-\gamma \tau}.
	\]
	
	In \cite{dutta2025some}, delay-independent stability has been established for the second quadrant ($k<0,\;\gamma>0$) and for the region $\gamma > 2k>0$ in the first quadrant, while instability has been shown for the third quadrant (see Fig. (\ref{ind1})). We now apply Theorem~\ref{thm:prelim} to analyze the stability of the equilibrium point $x_* = 0$ in the remaining regions, namely, the region $0 < \gamma < 2k$ and the fourth quadrant in the $k-\gamma$ plane.
	\begin{figure}
		\centering
		\includegraphics[scale=0.5]{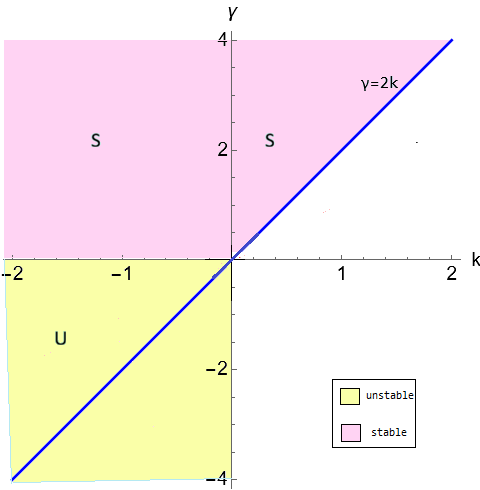}
		\caption{Delay Independent regions for $\tau_1=0$ and $\tau_2>0$}
		\label{ind1}
	\end{figure}
	
	Consider the first quadrant of $k-\gamma$ plane. The remaining region $0<\gamma<2k$  can be divided into two subregions: $0<\gamma<k$ and $k<\gamma<2k.$ Now, we analyse these subregions separately.\\
	
	Note that 
	\[
	\frac{db(\tau)}{d\tau} = k \gamma e^{-\gamma \tau} > 0,
	\]
	since $k > 0 \;\; \text{and} \;\; \gamma>0$ . Hence, the function $b(\tau)$ increases monotonically with respect to $\tau$ when $(k,\gamma)$ lies in the region $0<\gamma<2k$.

		\subsection{Case I: $0<\gamma<k$}
	In this case, we have $a=k-\gamma>0$. At $\tau=0$, we have
	$b(0)=-k<-a,$ which means that the initial point $(a,b(0))$ lies strictly in the SSR region of Theorem \ref{thm:prelim} (Fig. (\ref{fig:prel})). 
	As $\tau$ increases, the vertical array $T_1=\{(a,b(\tau)):\tau\ge0\}$ rises monotonically in $a-b$ plane (see Fig. (\ref{idea}), right half). Since $a>0$, the line $b=-a$ lies below the origin, and therefore the vertical line $T_1$ will eventually intersect $b=-a$ at some finite delay $\tau=\tau_*^{''}$ (as shown in Fig.(\ref{idea}), right half). 
	At this delay, the system moves from the SSR region into the unstable region of Theorem~\ref{thm:prelim}. Solving $b(\tau_{*}'')=-a$ gives
	\begin{equation}
		\label{eq:tau2}
		e^{-\gamma \tau_*^{''}} = \frac{k-\gamma}{k}
		\implies
		\tau_{*}{''} = \frac{-1}{\gamma} \log\left( \frac{k-\gamma}{k} \right).
	\end{equation}
	Note that  $0<\gamma<k$, implies $0<\frac{k-\gamma}{k}<1$ and hence $\tau_*^{''}>0$.
	\begin{figure}[H]
		\centering
		\includegraphics[scale=0.55]{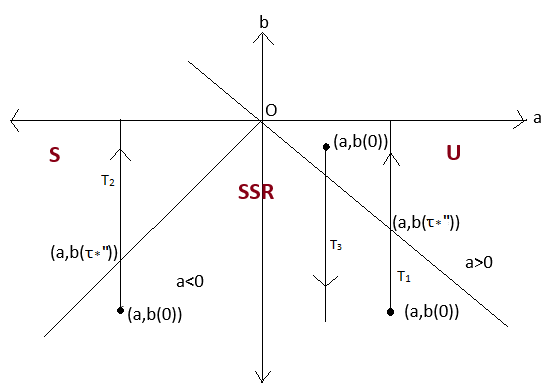} 
		\caption{Analysis for the region $0<\gamma<2k$}
		\label{idea}
	\end{figure}
	
	Independently, stability inside the SSR region is limited by the Hopf critical value $\tau_{*}'$ provided by Theorem  \ref{thm:prelim}. Note that, the parameter $b$ depends on $\tau.$ Hence, the expression (\ref{taustar}) for the critical delay also depends on $\tau$ viz. $\tau_*'(\tau)$. Because the expression for $\tau_{*}'(\tau)$ involves a square root, two distinct values may arise, corresponding to two branches of Hopf bifurcation. This gives rise to a complex  bifurcation scenario discussed below.

	\subsection{Bifurcation curves}
	
	The critical values corresponding to the Single Stable Region(SSR) are obtained from the intersection points of the curve $\tau_* = \tau_*{'}(\tau)$ and the line $\tau_* = \tau$ in the $\tau$–$\tau _*$ plane. 
	For fixed values of $k>0$ and $\gamma>0$, the curve $\tau_*'(\tau)$ takes complex values beyond a certain finite value of $\tau$. Therefore, the graph of $\tau_*{'}(\tau)$ will be sketched upto certain finite values only. Now, we will discuss various possibilities of intersection between the curves $\tau, \tau_*^{''} \;\text{and}\; \tau_*{'}(\tau)$ in the first quadrant of $k-\gamma$ plane. Further, we will provide bifurcation curves in $k-\gamma$ plane based on our observations.
	
	Based on the relative positions of $\tau_*^{'}$ and $\tau_*^{''}$, the following cases can occur:\\
	
	\textbf{Case 1:}  
	The first intersection point (say $\tau_{*a}^{'}$) occurs before $\tau_*^{''}$, and there is no second intersection. In this case, the equilibrium lies in the SSR region with the bifurcation value $\tau_* = \tau_{*a}'$ (see Fig.~\ref{case1}) i.e $0<\tau<\tau_{*a}' $ implies $x_*$ is asymptotically stable and $ \tau>\tau_{*a}{'}$ implies $x_*$ is unstable.
	\begin{figure}[H]
		\centering
		\includegraphics[scale=0.5]{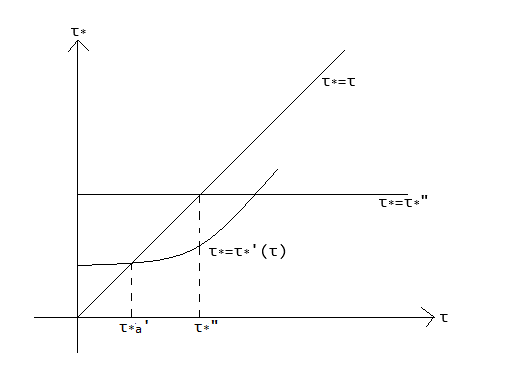}
		\caption{Single intersection leading to SSR behavior (Not upto scale).}
		\label{case1}
	\end{figure}
	\vspace{0.3cm}
	
	\textbf{Case 2:}  
	There is no intersection between the curve $\tau_* = \tau_*^{'}(\tau)$ and the line $\tau_* = \tau$ (see Fig.~\ref{case2}). Here also, the equilibrium lies in the SSR region with the critical value $\tau_* = \tau_*^{''}$. For $0<\tau<\tau_*^{''}$, we can observe that $\tau<\tau_*'(\tau)$, hence, $x_*$ is asymptotically stable. For $\tau>\tau_*^{''}$, the point $(a,b(\tau))$ enters in the unstable region and $x_*$ becomes unstable.
		\begin{figure}[H]
		\centering
		\includegraphics[scale=0.5]{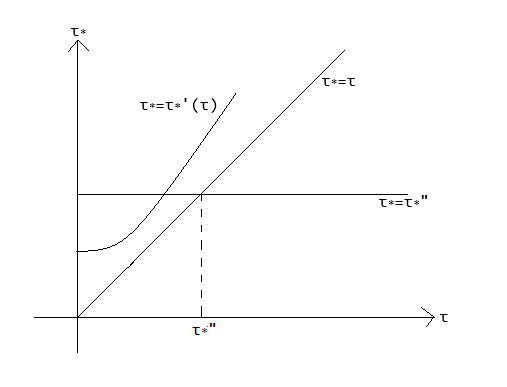}
		\caption{No intersection between the curve $\tau_* = \tau_*^{'}(\tau)$ and the line $\tau_* = \tau$ leading to SSR behavior  (Not upto scale).}
		\label{case2}
	\end{figure}
\vspace{0.3cm}	
	
	\textbf{Case 3:}  
	In this case, the first intersection $\tau_{*a}{'}$ between $\tau$ and $\tau_*{'}(\tau)$  occurs before $\tau_*^{''}$, and the second intersection $\tau_{*b}'$ occurs after $\tau_*^{''}$ (see Fig.~\ref{case3}). The second intersection does not affect stability because it lies beyond $\tau_* ''$. Hence, the equilibrium lies in the SSR region with the critical value $\tau_* = \tau_{*a}{'}$.
		\begin{figure}[H]
		\centering
		\includegraphics[scale=0.5]{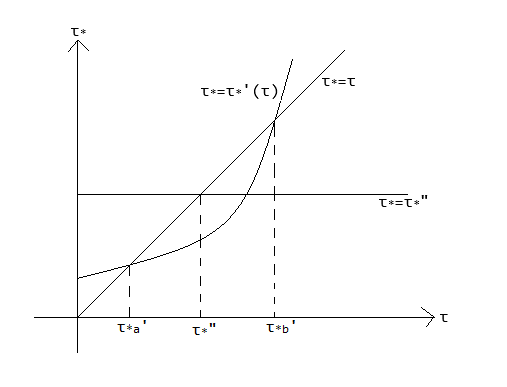}
		\caption{Two intersections with only the first one affecting stability (Not upto scale).}
		\label{case3}
	\end{figure}
	
	\vspace{0.3cm}
	
	\textbf{Case 4:}  
	Both intersection points (say $\tau_{*a}{'}$ and $\tau_{*b}{'}$) occur before $\tau_*^{''}$ (see Fig.~\ref{case4}). This case gives rise to a switch of the form Stable–Unstable–Stable–Unstable (SUSU):
	\begin{itemize}
		\item stable for $0<\tau<\tau_{*a}{'}$
		\item unstable for $\tau_{*a}{'}<\tau<\tau_{*b}{'}$
		\item stable for $\tau_{*b}{'}<\tau<\tau_*{''}$
		\item unstable for $\tau>\tau_*{''}$
	\end{itemize}
	
	\begin{figure}[H]
		\centering
		\includegraphics[scale=0.5]{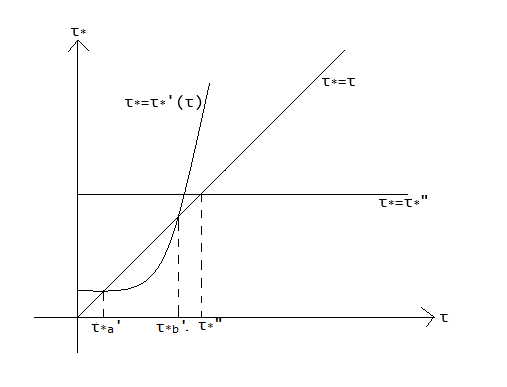}
		\caption{Both intersections before $\tau_*^{''}$ leading to SUSU behavior  (Not upto scale).}
		\label{case4}
	\end{figure}
	
	\vspace{0.3cm}
	
	\textbf{Case 5:}  
	Both intersection points (say $\tau_{*a}{'}$ and $\tau_{*b}{'}$) occur after $\tau_*^{''}$ (see Fig.~\ref{case5}). In this situation, neither intersection affects stability. Hence, the equilibrium remains in the SSR region with the critical value $\tau = \tau_*{''}$. Euilibrium is stable for $0<\tau<\tau_*^{''}$ and unstable for $\tau>\tau_*{''}$.
	
	\begin{figure}[H]
		\centering
		\includegraphics[scale=0.5]{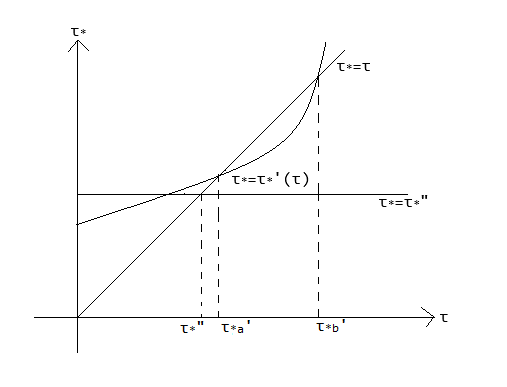}
		\caption{Both intersections after $\tau_*^{''}$—stability unaffected (SSR region) (Not upto scale).}
		\label{case5}
	\end{figure}
		Based on these different scenarios, we have two bifurcation curves in the first quadrant of $k-\gamma$ plane, say $\gamma=h_1(k)$ and $\gamma=h_2(k)$.
	
	\begin{itemize}
		\item The curve $\gamma=h_1(k)$ bifurcates the two regions: 
		\begin{itemize}
			\item when there are two intersection between the curves $\tau_*=\tau_*'(\tau)$ and $\tau_*=\tau$ and
			\item There is no intersection between these curves.
			
		\end{itemize}
		So, it will be plotted by using the condition that the line $\tau _ *=\tau$ is a tangent to the curve $\tau_*=\tau_*{'}(\tau)$. This tangency condition will be given by the points $(k,\gamma)$ for which  
		\[ \tau_*{'}(\tau)-\tau=0 \;\; \text{and} \;\; \frac{\partial\tau_*{'}(\tau)}{\partial\tau}-1=0.\]
		For a fixed value of $\alpha,$ we use ``Table" and ``FindRoot" commands in Wolfram Mathematica to plot this curve in parameter plane (ref. Fig. (\ref{kgamma})).
		\begin{figure}[H]
			\centering
			\includegraphics[scale=0.44]{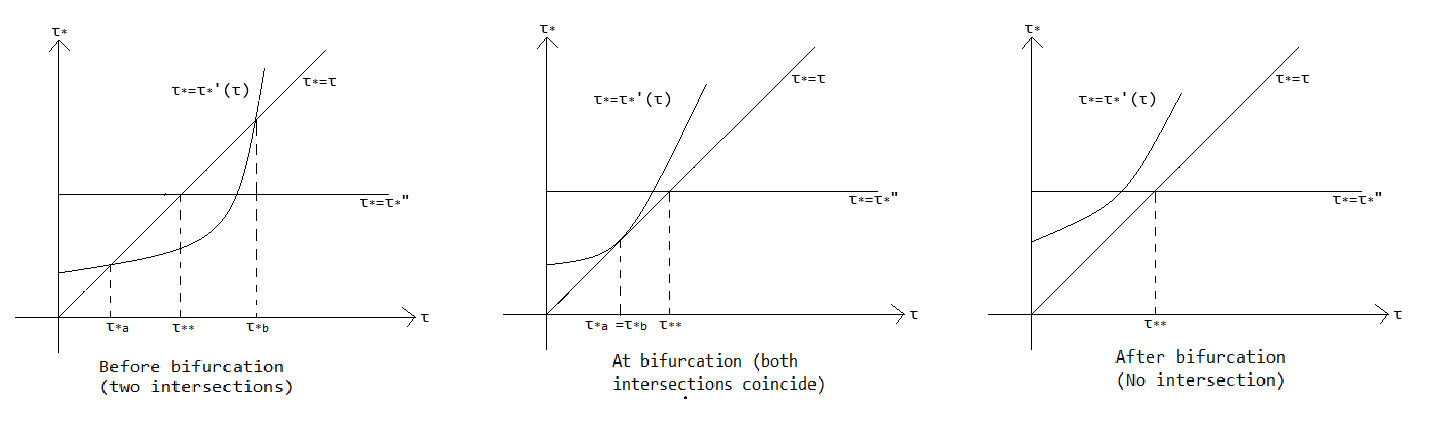}
			\caption{Geometric interpretation of curve $\gamma=h_1(k)$}
			\label{h1}
		\end{figure}

		\item The curve $\gamma=h_2(k)$ bifurcates the two regions: 
		\begin{itemize}
			\item $\tau_*{'}(\tau_*{''})>\tau_*{''}$ and
			\item $\tau_*{'}(\tau_*{''})<\tau_*{''}$ (see Fig. (\ref{h1})).
		\end{itemize}
		Hence, it will be plotted by using the condition 
		\[\tau_*{'}(\tau_*{''})=\tau_*{''}\]
		
		Again, for a fixed value of $\alpha,$ we use ``Table" and ``FindRoot" command in Mathematica to plot this curve (ref. Fig. (\ref{kgamma})) in $k-\gamma$ plane.
		\begin{figure}[H]
			\centering
			\includegraphics[scale=0.4]{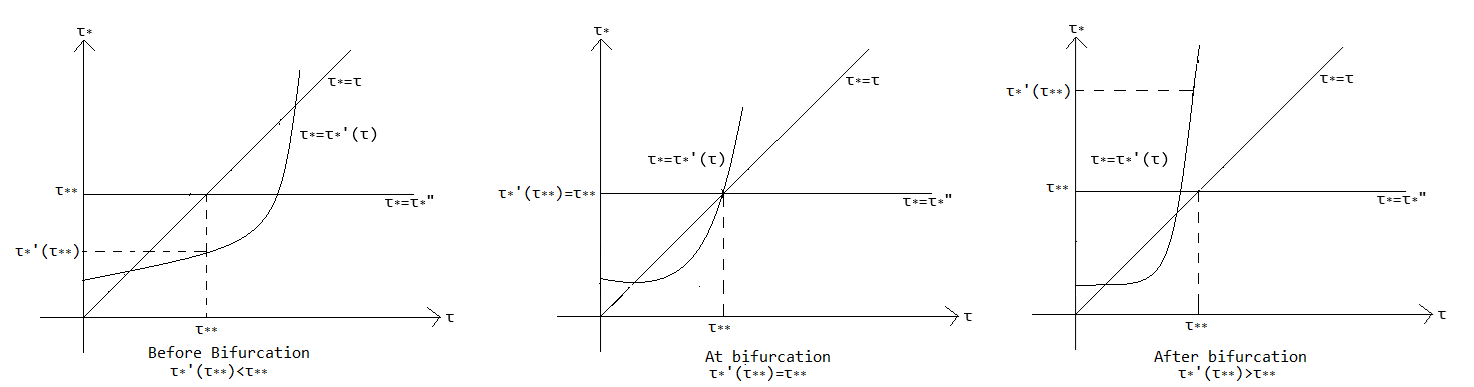}
			\caption{Geometric interpretation of curve $\gamma=h_2(k)$}
			\label{h1_a}
		\end{figure}
	\end{itemize}

	\subsection{Case II: $k<\gamma<2k$ }
	Here, $a=k-\gamma<0$. At $\tau=0$, $b(0)=-k < a$.
	
	So, once again the initial point $(a,b(0))$ lies in the SSR region of Theorem~\ref{thm:prelim}. Thus, the system starts with local stability for sufficiently small delay. 
	
	Here also, as $\tau$ increases, $b(\tau)$ goes toward $0$. Again, the arrow $T_2=\{(a,b(\tau))|\tau\ge0\}$ will intersect the line $b=a$ at some finite delay $\tau=\tau_{*}{''}$ (ref. Fig. (\ref{idea}) left half). At this critical value, the equilibrium crosses from the SSR region to the stable region as shown in Theorem \ref{thm:prelim}. The crossing delay $\tau_{*}{''}$ is obtained from \[ b(\tau_{*}{''})=a \]
	\begin{equation}
		\label{eq:tau2_case2}
		\text{i.e} \;\;\;-k e^{-\gamma \tau_{*}{''}} = k-\gamma
		\implies
		\tau_{*}{''} = \frac{-1}{\gamma} \log\left( \frac{\gamma-k}{k} \right)>0.
	\end{equation}
	Here $\tau_{*}''$ marks entry into the stable region rather than the unstable region. 
	
	Again, because of the SSR region, there is a critical value $\tau_{*}'(\tau)$. Here also, we plot the same bifurcation curves $h_1(k)$ and $h_2(k)$. However, in this region, the curve $h_2(k)$ does not appear; only $h_1(k)$ exists. 
	On the curve $h_2(k)$, we have $b = a$. 
	If we put this condition in the expression of $\tau_{*}'$ given in Theorem~ ~(\ref{thm:prelim}), it becomes undefined. 
	The curve $h_1(k)$ divides the plane into two parts: the S–U–S region (two intersections) and the stable region (no intersection).
	Thus, the bifurcation curves $\gamma=h_1(k)$ and $\gamma=h_2(k)$ together with the lines $\gamma=k$ and $\gamma=2k$ partition the entire region $0<\gamma<2k$ into six distinct regions (see Fig. (\ref{kgamma})). 
	
	\begin{figure}[H]
		\centering
		\includegraphics[scale=0.6]{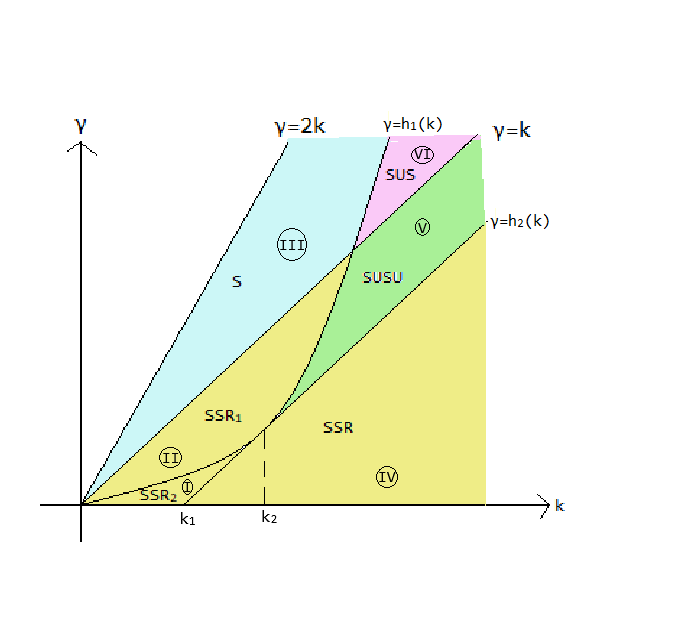}
		\caption{Stability regions of $0<\gamma<2k$ for $\alpha=0.4$}
		\label{kgamma}
	\end{figure}

	We have the following observations from this figure:
	\begin{itemize}
		\item In region (I), we have the scenarios as shown in Fig.(\ref{region1}):
\begin{figure}[H]
	\centering
	\begin{subfigure}[t]{0.49\textwidth}
		\centering
		\includegraphics[height=5cm, trim=50 20 50 10, clip]{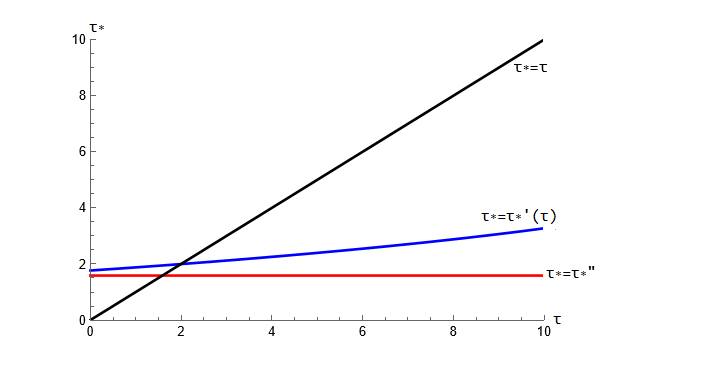}
		\caption{$\alpha=0.4,\; k=0.65,\; \gamma=0.04$}
	\end{subfigure}
	\hspace{0.12cm}
	\begin{subfigure}[t]{0.49\textwidth}
		\centering
		\includegraphics[height=5cm, trim=60 30 50 10, clip]{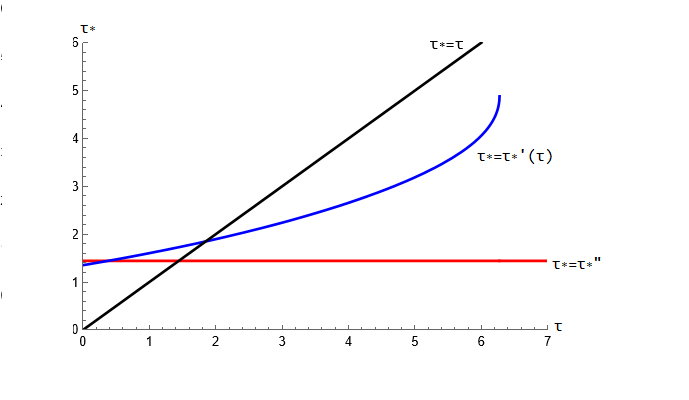}
		\caption{$\alpha=0.4,\; k=0.75,\; \gamma=0.11$}
	\end{subfigure}
	
	\caption{SSR behavior in region I for some set of parameter values}
	\label{region1}
\end{figure}
		Here, the intersection between the curve $\tau_*=\tau_*{'}(\tau)$ and $\tau=\tau_*$ is after $\tau_{*}{''}.$ So, there will be a SSR behavior with critical value $\tau_{*}{''} = \frac{-1}{\gamma} \log\left( \frac{k-\gamma}{k} \right).$
		
		\item In region (II), the following cases will be there (see Fig. (\ref{region II})):
		
		\begin{figure}[H]
			\centering
			% --- First row ---
			\begin{subfigure}{0.45\textwidth}
				\includegraphics[height=5cm, clip]{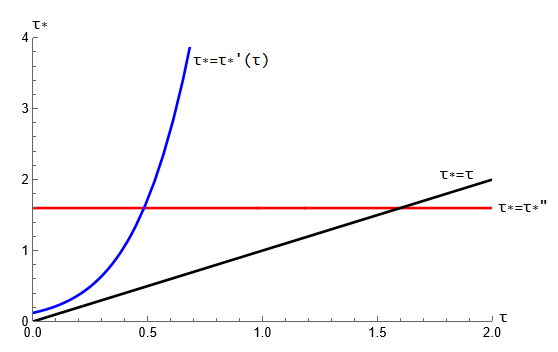}
				\caption{$\alpha=0.4,\; k=1.13,\; \gamma=0.83$}
			\end{subfigure}%
			\hfill
			\begin{subfigure}{0.45\textwidth}
				\includegraphics[height=5cm, clip]{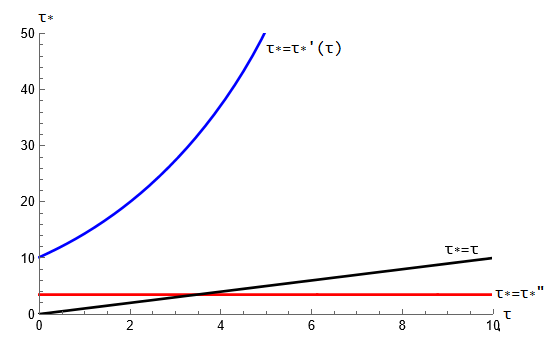}
				\caption{$\alpha=0.4,\; k=0.4,\; \gamma=0.2$}
			\end{subfigure}
			
			\vspace{0.5em}
			\caption{SSR behavior in region II for some set of parameter values}
			\label{region II}
		\end{figure}
		In both cases, $\tau$ is always below $\tau_*{'}(\tau)$. In this region, there is no intersection between the curve $\tau_*=\tau_{*}{'}(\tau)$ and $\tau=\tau_*$. So, again,  there will be an SSR behavior with critical value $\tau_{*}{''} = \frac{-1}{\gamma} \log\left( \frac{k-\gamma}{k} \right)$. Note that, the intersection between the curves $\tau=\tau*{'}(\tau)$ and $\tau=\tau_*{''}$ will not affect the qualitative properties of the system.
		
		\item In region (III), only one scenario is there, as shown in Fig.(\ref{region III}):
		\begin{figure}[H]
			\centering
			\includegraphics[scale=0.4]{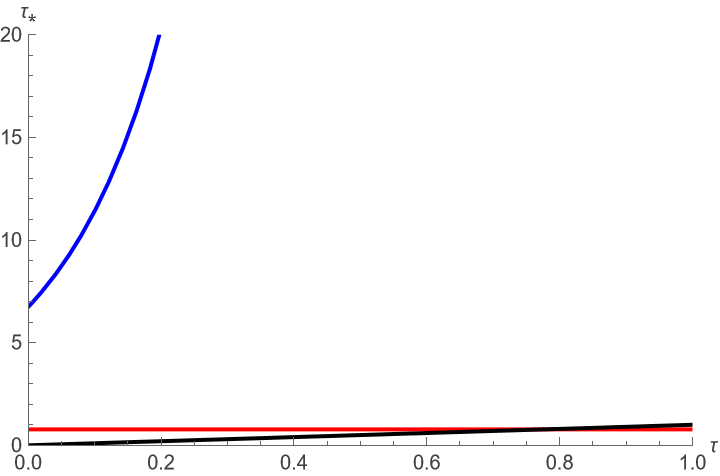}
			\caption{Stable behavior in region III with $\alpha=0.4,\; k=1,\; \gamma=1.35$}
			\label{region III}
		\end{figure}

		In this region, there is no intersection between the curve $\tau_*=\tau_{*}{'}(\tau)$ and $\tau=\tau_{*}{''}$, i.e, there is no critical value in the SSR region. Since $\gamma>k,$ i.e, $a<0$, the system will remain stable only. 
		
		\item In region (IV), we have the following cases (see Fig.(\ref{region IV})):
		\begin{figure}[H]
			\centering
			% --- First row ---
			\begin{subfigure}{0.45\textwidth}
				\includegraphics[height=5.5cm, trim=50 30 80 10, clip]{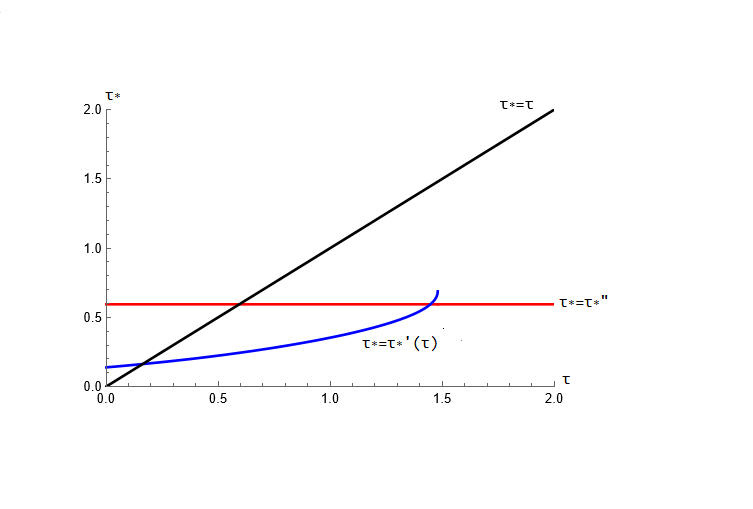}
				\caption{$\alpha=0.4,\;k=2,\;\gamma=0.6$}
			\end{subfigure}
			\hspace{0.2cm}
			\begin{subfigure}{0.45\textwidth}
				\includegraphics[height=5cm, trim=50 20 80 10, clip]{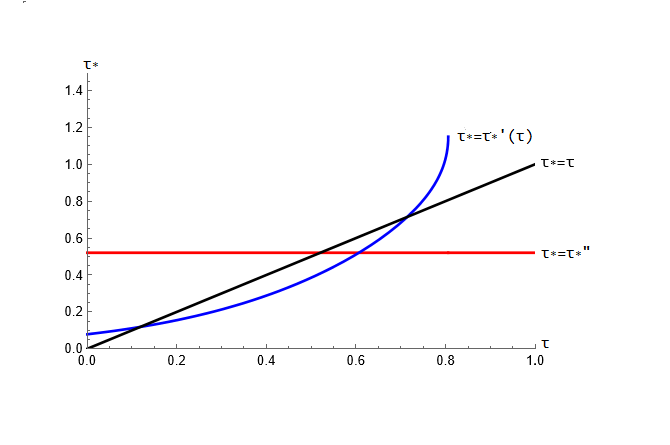}
				\caption{$\alpha=0.4,\;k=3,\;\gamma=1.86$}
			\end{subfigure}
			
			\vspace{0.5em}
			\caption{SSR behavior in region IV wih given values of parameters}
			\label{region IV}
		\end{figure}
		
		In this region, there is always one intersection between the curve $\tau_*=\tau_{*}{'}(\tau)$ and $\tau_*=\tau$ (say $\tau_{*a}{'}$) which is before $\tau_{*}{''}$. So, equilibrium will have SSR behavior with critical value as $\tau_{*a}{'}.$
		
		\item In region (V), only one behavior is observed (see Fig . (\ref{region V})):
		\begin{figure}[H]
			\centering
			\includegraphics[scale=0.4]{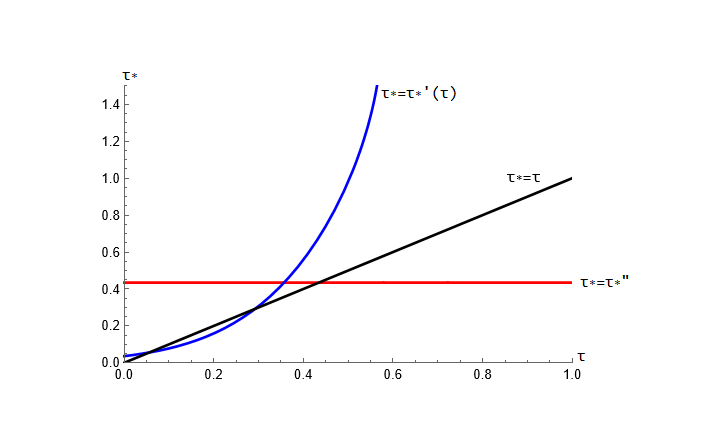}
			\caption{SUSU behavior in region V with $\alpha=0.4,\;k=4.62,\;\gamma=3.69$}
			\label{region V}
		\end{figure}
		It can be observed that there are two intersections (say $\tau_{*a}{'} \;\; \text{and}\;\; \tau_{*b}{'}$) before $\tau_{*}{''}.$ So, there will be SUSU (Stable-Unstable-Stable-Unstable) behavior in this region, i.e
		\begin{itemize}
			\item stable for $0<\tau<\tau_{*a}{'}$,
			\item unstable for $\tau_{*a}{'}<\tau<\tau_{*b}{'}$,
			\item  stable for $\tau_{*b}{'}<\tau<\tau_{*}{''}$ and
			\item  unstable for $\tau>\tau_{*}{''}$.
		\end{itemize}

		\item In region (VI), we get a figure similar to Fig.(\ref{region V}):
		\begin{figure}[H]
			\centering
			\includegraphics[scale=0.4]{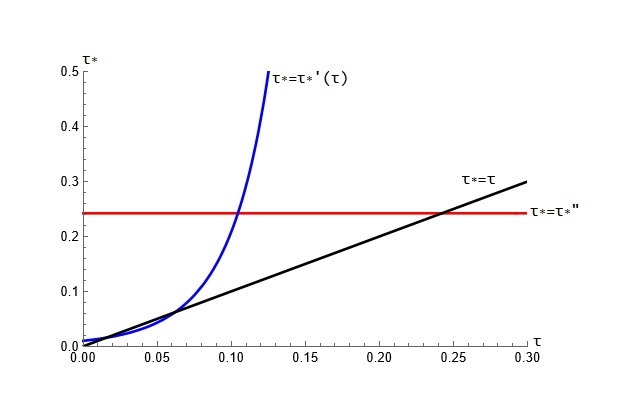}
			\caption{SUS behavior in region VI with $\alpha=0.4,k=9.8,\;\gamma=10.56$}
			%   \label{fig:tau_series}%
		\end{figure}
		However, $a<0$ and hence $\tau>\tau_*''$ gives stable solutions.
		Since, $\gamma>k$ in this region and  there are two intersections (say $\tau_{*a}' \;\; \text{and}\;\; \tau_{*b}'$) before $\tau_{*}''$. So, there will be SUS (Stable-Unstable-Stable) behavior in this region i.e
		\begin{itemize}
			\item stable for $0<\tau<\tau_{*a}{'}$,
			\item unstable for $\tau_{*a}{'}<\tau<\tau_{*b}{'}$ and
			\item  stable for $\tau>\tau_{*b}{'}$ 
		\end{itemize}
	\end{itemize}
		\subsection{Fourth Quadrant Analysis}
	
	In this subsection, we study the behaviour of the system in the fourth quadrant 
	of the $k$--$\gamma$ plane. Here, we have $a = k - \gamma > 0$ and $b(\tau) < 0$. 
	Also note that $b'(\tau) = k\gamma e^{-\gamma\tau} < 0 \;\; \text{since } \gamma < 0$, which means that $b(\tau)$ keeps decreasing as $\tau$ increases.
	
	At $\tau = 0$, we have $b(0) = -k<0$, and since $\gamma < 0$, this gives  
	$b(0) = -k > \gamma - k = -a$.  
	Thus, the equilibrium point is initially in the unstable region (see the arrow  $T_3$ in Fig. (\ref{idea})).  
	As $\tau$ increases, the value of $b(\tau)$ decreases further and eventually enters the 
	SSR region.
	
	In this quadrant, two important values of $\tau$ appear:
	\begin{itemize}
		\item $\tau_{*}{''}$, which is obtained from the boundary condition $b(\tau_{*}{''}) = -a$,
		\item $\tau_{*}{'}$, which comes from the condition of the SSR region.
	\end{itemize}
	
	To understand which behaviour appears, we compare the curve $\tau_{*}{'}(\tau)$ with the 
	line $\tau_{*} = \tau$. We observed that, for the parameter values $k>0$ and $\gamma<0$, the function $\tau_{*}{'}(\tau)$ decreases steadily 
	and approaches zero as $\tau$ increases. Therefore, there is always exactly one intersection 
	between the curves $\tau_{*} = \tau$ and $\tau_{*} = \tau_{*}{'}(\tau)$. Because of this, the 
	curve $\gamma = h_{1}(k)$ does not arise in this quadrant. Only the bifurcation curve 
	$\gamma=h_{2}(k)$ appears, and it is obtained from the condition 
	$\tau_{*}{'}(\tau_{*}{''}) = \tau_{*}{''}$.
	
	This curve $\gamma=h_{2}(k)$ separates two kinds of behaviour:
	\begin{itemize}
		\item \textbf{Case 1:} The intersection between $\tau_{*}{'}(\tau)$ and $\tau$ occurs 
		before $\tau_{*}{''}$.  
		In this case, the stability does not change, and the equilibrium point remains unstable 
		for all $\tau$.
		
		\item \textbf{Case 2:} The intersection (say at $\tau_{*a}{'}$) occurs after
		$\tau_{*}{''}$ (Fig.(18(a)).  
		Here, the system enters the SSR region for some interval of $\tau$, leading to a 
		USU behaviour (Unstable–Stable–Unstable), as illustrated in (Fig.(18(c)).  
		That is:
		\begin{itemize}
			\item the equilibrium is unstable for $0 < \tau < \tau_{*}{''}$,
			\item becomes stable for $\tau_{*}{''} < \tau < \tau_{*a}{'}$, and
			\item becomes unstable again when $\tau > \tau_{*a}{'}$.
		\end{itemize}
	\end{itemize}
	
	\begin{figure}[H]
		\centering
		\includegraphics[scale=0.4]{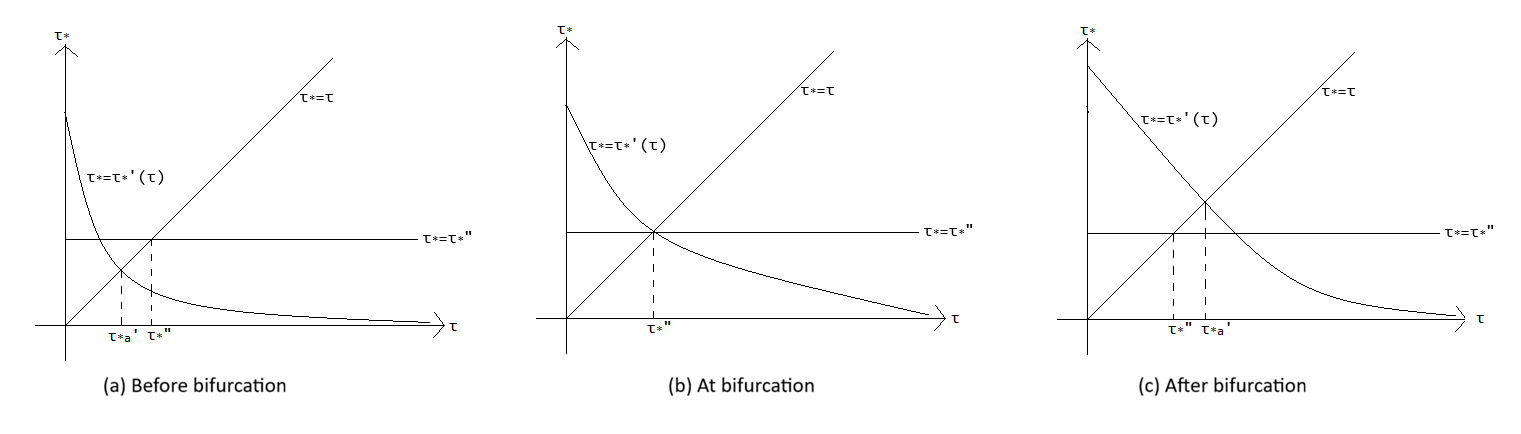}
		\caption{Geometric interpretation of the curve $\gamma = h_{2}(k)$ in the fourth quadrant}
		\label{h2}
	\end{figure}
	
	\begin{figure}[H]
		\centering
		\includegraphics[scale=0.7]{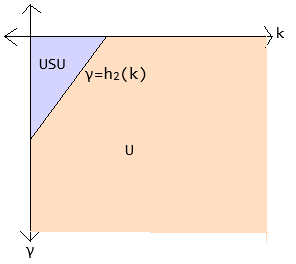}
		\caption{Stability behaviour in the fourth quadrant}
		\label{c14q}
	\end{figure}
\subsection{Examples for $\tau_1 = 0$}

We present illustrative examples for the subregions in which the critical delay value depends on $\alpha$. Here, we fix $\alpha = 0.4$.

\textbf{Example 4.1:} Let $k = 2$ and $\gamma = 0.6$. These parameter values lie in subregion IV (see Fig. \ref{kgamma}), which corresponds to a Single Stable Region (SSR) with critical delay $\tau_{*a'}$. The critical value $\tau_{*a'} = 0.1630$ is obtained as the first intersection point of the curves $\tau_*= \tau_*'(\tau)$ and $\tau_* = \tau$.

For $\tau = 0.14$, the equilibrium is stable (see Fig. \ref{fig3_4_1}), whereas it becomes unstable for $\tau = 0.52$ and $\tau = 0.67$ (see Figs. \ref{fig3_4_2} and \ref{fig3_4_3}).
\begin{figure}[H]
	\centering
	\begin{subfigure}{0.32\textwidth}
		\centering
		\includegraphics[width=\linewidth]{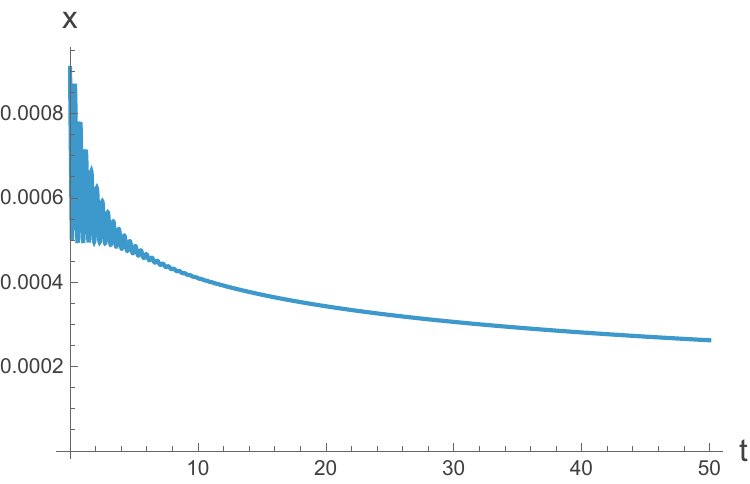}
		\caption{$\tau=0.14$}
		\label{fig3_4_1}
	\end{subfigure}
	\hfill
	\begin{subfigure}{0.32\textwidth}
		\centering
		\includegraphics[width=\linewidth]{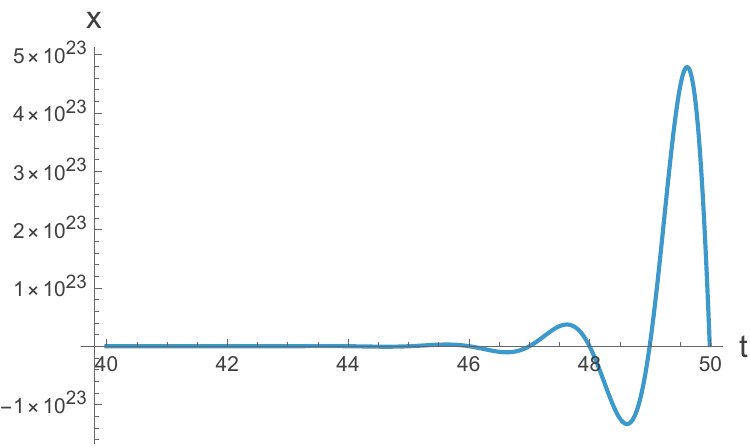}
		\caption{$\tau = 0.52$}
		\label{fig3_4_2}
	\end{subfigure}
	\hfill	
	\begin{subfigure}{0.32\textwidth}
		\centering
		\includegraphics[width=\linewidth]{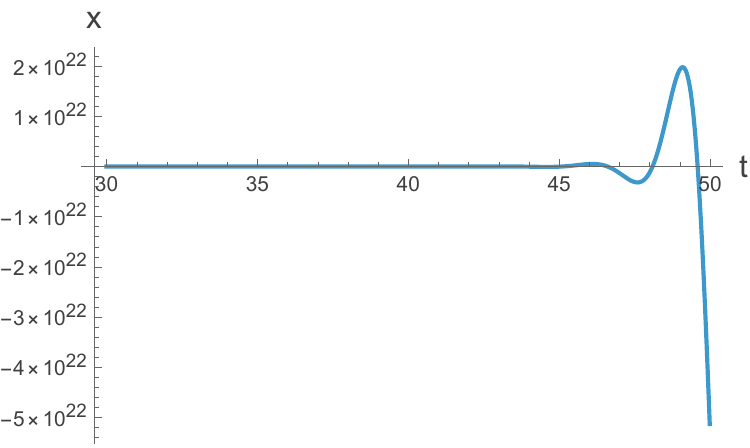}
		\caption{$\tau = 0.67$}
		\label{fig3_4_3}
	\end{subfigure}
	\caption{SSR behavior for region IV}
\end{figure}
\textbf{Example 4.2:} Consider the parameter values $k = 4.62$ and $\gamma = 3.69$. These values lie in subregion V (see Fig.~\ref{kgamma}), which exhibits a Stable-Unstable-Stable--Unstable (SUSU) switching behavior. This behavior is characterized by the critical delay values $\tau_{*a'}$, $\tau_{*b'}$, and $\tau_*^{''}$.

The critical values $\tau_{*a'} = 0.0560$ and $\tau_{*b'} = 0.2925$ correspond to the first and second intersection points of the curves $\tau_*'(\tau)$ and $\tau_* = \tau$, respectively. The third critical value $\tau_*^{''} = 0.4344$ is obtained from the condition $b(\tau) = -a$.

The stability of the equilibrium changes as $\tau$ varies. For $\tau = 0.04$, the equilibrium is stable (see Fig.~\ref{fig3_4_4}). When $\tau = 0.25$ (see Fig.~\ref{fig3_4_5}), the system becomes unstable. Increasing $\tau$ further to $0.38$ (see Fig.~\ref{fig3_4_6}) restores stability. However, for $\tau = 0.56$ (see Fig.~\ref{fig3_4_7}), the equilibrium again loses stability and remains unstable thereafter.
\begin{figure}[H]
	\centering
	\begin{subfigure}{0.24\textwidth}
		\centering
		\includegraphics[width=\linewidth]{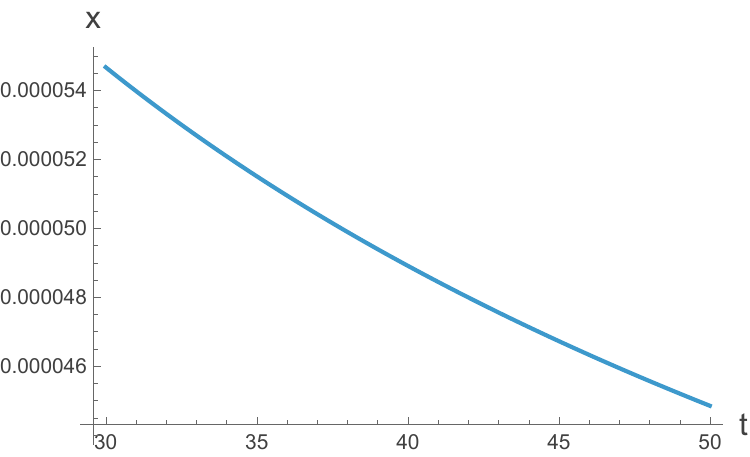}
		\caption{$\tau=0.04$}
		\label{fig3_4_4}
	\end{subfigure}
	\hfill
	\begin{subfigure}{0.24\textwidth}
		\centering
		\includegraphics[width=\linewidth]{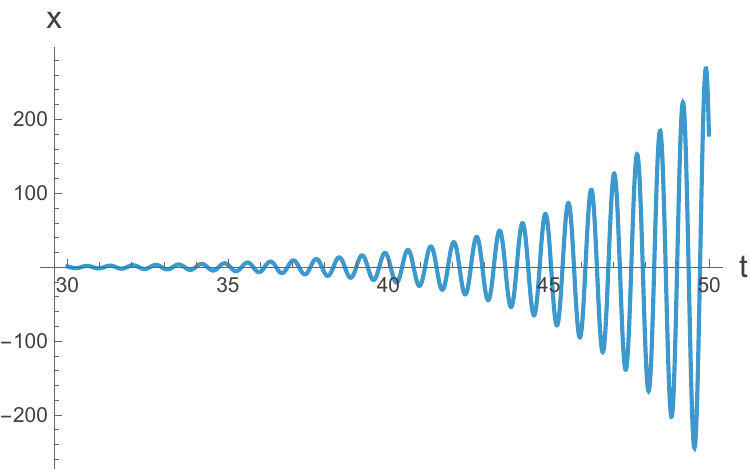}
		\caption{$\tau = 0.25$}
		\label{fig3_4_5}
	\end{subfigure}
	\hfill
	\begin{subfigure}{0.24\textwidth}
		\centering
		\includegraphics[width=\linewidth]{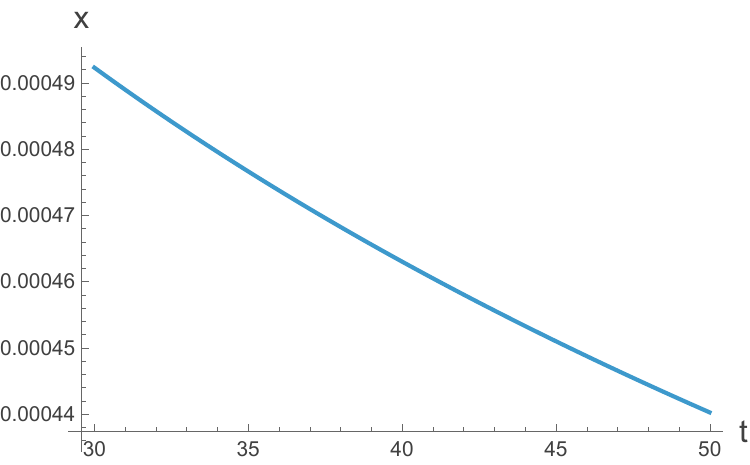}
		\caption{$\tau = 0.38$}
		\label{fig3_4_6}
	\end{subfigure}
	\hfill
	\begin{subfigure}{0.24\textwidth}
		\centering
		\includegraphics[width=\linewidth]{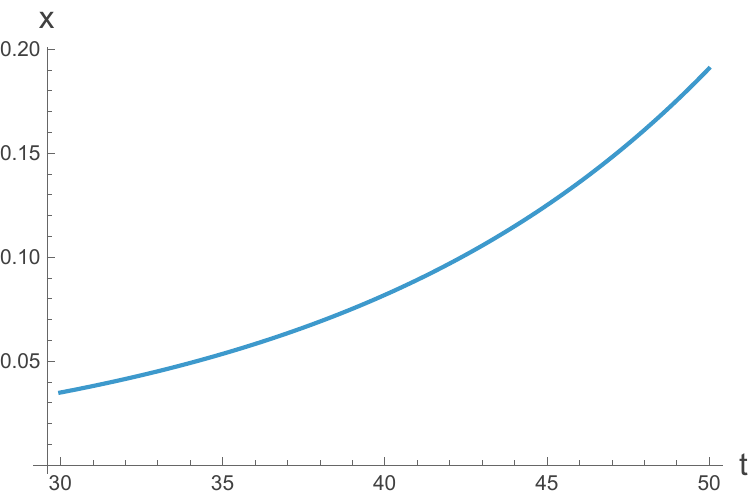}
		\caption{$\tau = 0.56$}
		\label{fig3_4_7}
	\end{subfigure}
	\caption{SUSU switching behavior in subregion V}
\end{figure}

\textbf{Example 4.3:} For $k = 9.8$ and $\gamma = 10.56$, the parameter pair falls within subregion VI (see Fig. \ref{kgamma}). In this region, the system demonstrates a Stable–Unstable–Stable (SUS) transition pattern governed by two critical delay thresholds, $\tau_{*a'}$ and $\tau_{*b'}$. Numerically, these are $\tau_{*a'} = 0.0157$ and $\tau_{*b'} = 0.0619$, obtained as the first two points of intersection between the curves $\tau_* = \tau_*'(\tau)$ and $\tau_ *= \tau$.

The effect of the delay parameter $\tau$ on stability can be observed through representative values. The equilibrium remains stable at $\tau = 0.01$ (Fig. \ref{fig3_4_8}). As $\tau$ increases past the first threshold, instability arises, as seen at $\tau = 0.05$ (Fig. \ref{fig3_4_9}). Upon further increase in $\tau$, stability is recovered; this is illustrated at $\tau = 0.25$ (Fig. \ref{fig3_4_10}).

\begin{figure}[H]
	\centering
	\begin{subfigure}{0.3\textwidth}
		\centering
		\includegraphics[width=\linewidth]{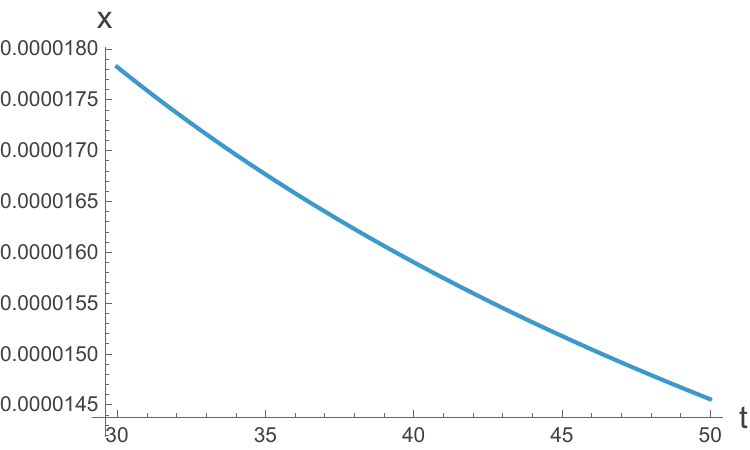}
		\caption{$\tau=0.01$}
		\label{fig3_4_8}
	\end{subfigure}
	\hfill
	\begin{subfigure}{0.3\textwidth}
		\centering
		\includegraphics[width=\linewidth]{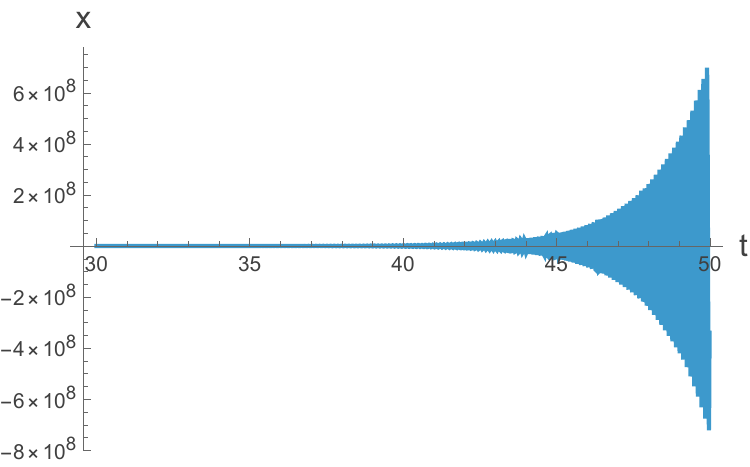}
		\caption{$\tau = 0.05$}
		\label{fig3_4_9}
	\end{subfigure}
	\hfill
	\begin{subfigure}{0.3\textwidth}
		\centering
		\includegraphics[width=\linewidth]{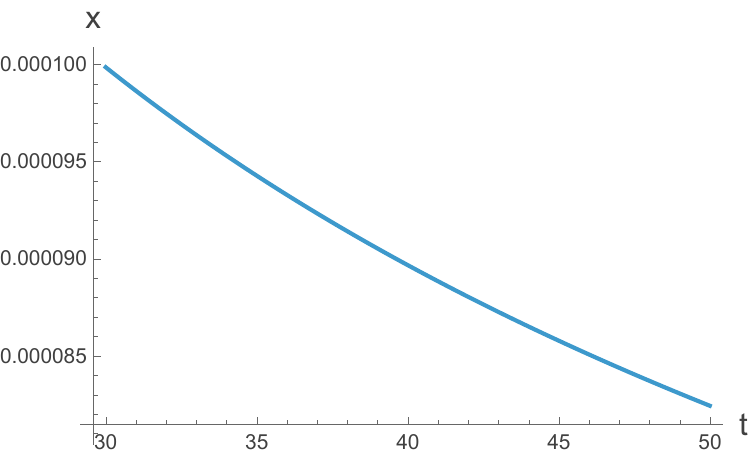}
		\caption{$\tau = 0.25$}
		\label{fig3_4_10}
	\end{subfigure}
	
	\caption{SUS switching behavior in subregion VI}
\end{figure}

 \textbf{Example 4.4:} 
 Let $k = 0.23$ and $\gamma = -0.12$. These parameter values lie in the fourth quadrant of the $k$--$\gamma$ plane (see Fig.~\ref{c14q}) and correspond to an Unstable-Stable-Unstable (USU) switching behavior.
 
 In this case, two critical delay values arise:
 $
 \tau_*^{''} = 3.4987, \; \tau_{*a}' = 5.4386,
 $
 where $\tau_*^{''}$ is obtained from the condition $b(\tau) = -a$, and $\tau_{*a}'$ is the point of intersection of the curves $\tau_* = \tau_*'(\tau)$ and $\tau_* = \tau$, occurring after $\tau_*^{''}$.
 
 The stability of the equilibrium changes as $\tau$ varies. For $\tau = 3.35 < \tau_*^{''}$, the equilibrium is unstable (see Fig.~\ref{fig3_4_11}). For $\tau = 5.3$, satisfying $\tau_*^{''} < \tau < \tau_{*a}'$, the equilibrium becomes stable (see Fig.~\ref{fig3_4_12}). For $\tau = 5.8 > \tau_{*a}'$, the equilibrium again becomes unstable (see Fig.~\ref{fig3_4_13}).
 
 \begin{figure}[H]
 	\centering
 	\begin{subfigure}{0.3\textwidth}
 		\centering
 		\includegraphics[width=\linewidth]{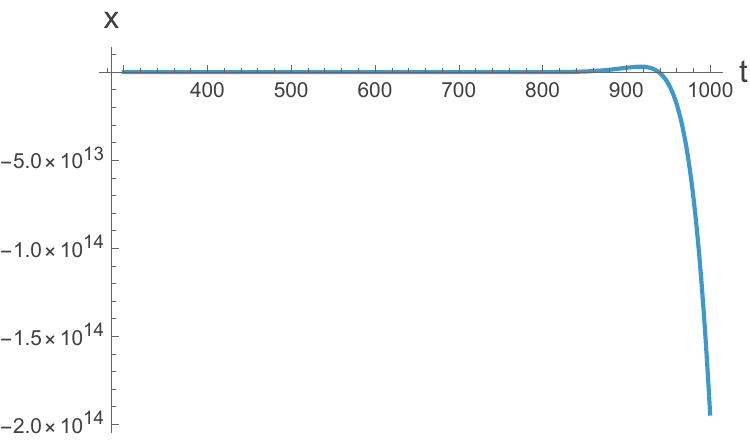}
 		\caption{$\tau = 3.35$}
 		\label{fig3_4_11}
 	\end{subfigure}
 	\hfill
 	\begin{subfigure}{0.3\textwidth}
 		\centering
 		\includegraphics[width=\linewidth]{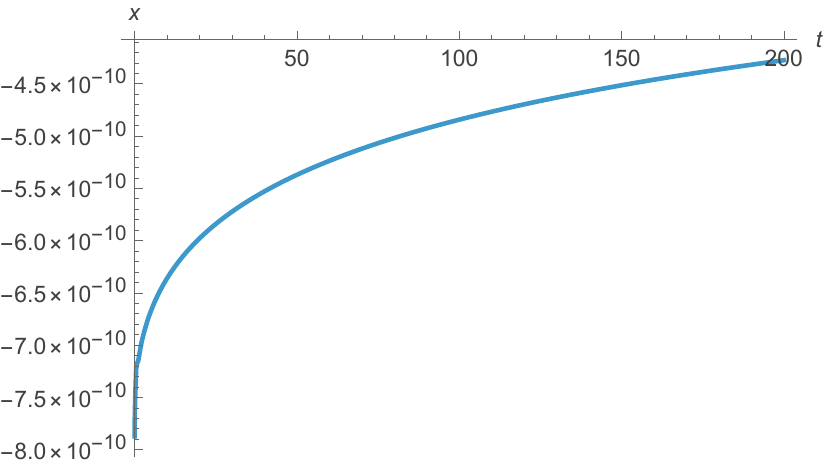}
 		\caption{$\tau = 5.3$}
 		\label{fig3_4_12}
 	\end{subfigure}
 	\hfill
 	\begin{subfigure}{0.3\textwidth}
 		\centering
 		\includegraphics[width=\linewidth]{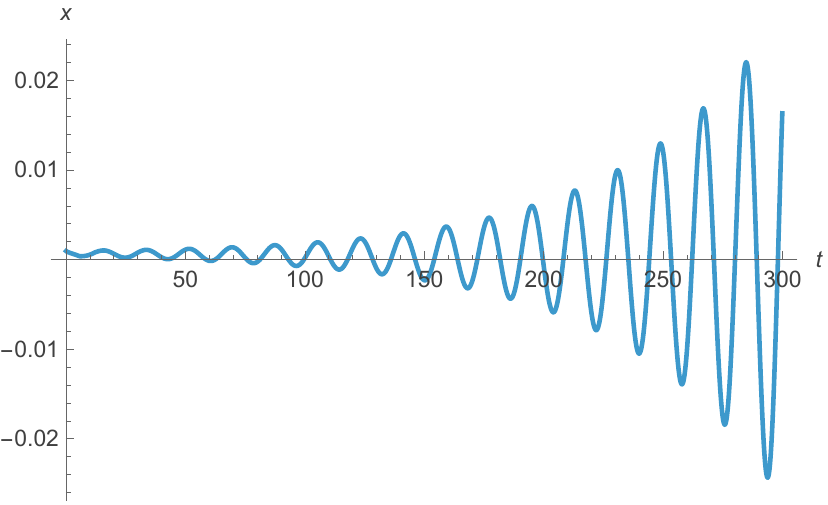}
 		\caption{$\tau = 5.8$}
 		\label{fig3_4_13}
 	\end{subfigure}
 	
 	\caption{USU switching behavior in the fourth quadrant}
 \end{figure}

	\section{Case 2: $\tau_1>0,\;\tau_2\ge0$}
		\label{sec:C2}
	In this case, both delays are retained. The linearized version of this equation is given by (\ref{eq:linear}).
	The corresponding characteristic equation is given by $(\ref{eq:char})$.

	We now present some delay-independent results summarizing stability behaviors for this case.
\begin{theorem}
	\label{thmcase2_unified}
	The equilibrium point $x_* = 0$ of Eq.~(\ref{eq:main}) i.e 
		\begin{equation}
		D^{\alpha} x(t)
		= -\gamma x(t)
		+ g\big(x(t - \tau_1)\big)
		- e^{-\gamma \tau_2}\, g\big(x(t - \tau_1 - \tau_2)\big),
		\qquad 0 < \alpha \le 1
	\end{equation}
	 exhibits the following stability behavior:
	
	\textbf{Delay-Independent Instability:}
	\begin{enumerate}[label=(\roman*)]
		\item If $0 < \gamma < k$, then for 
		\[
		\tau_2 > -\frac{1}{\gamma} \log\left(\frac{k - \gamma}{k}\right),
		\]
		the equilibrium is unstable for all $\tau_1 \ge 0$ and $0 < \alpha \le 1$.
		
		\item If $\gamma < 0 < k$, then for 
		\[
		\tau_2 < -\frac{1}{\gamma} \log\left(\frac{k - \gamma}{k}\right),
		\]
		the equilibrium is unstable for all $\tau_1 \ge 0$ and $0 < \alpha \le 1$.
	\end{enumerate}
	
\end{theorem}
\begin{proof}
	Consider the linearized version of  fractional delay differential equation (\ref{eq:main}) i.e 
	\begin{equation}
		D^{\alpha} x(t)
		= -\gamma x(t)
		+ k\,x(t - \tau_1)
		- k e^{-\gamma \tau_2}\, x(t - \tau_1 - \tau_2),
	\end{equation}
	where $0<\alpha\le 1$. The stability of the equilibrium $x_*=0$ is determined by the roots of the corresponding characteristic equation i.e
	\[
	\lambda^{\alpha}
	= -\gamma + k e^{-\lambda \tau_1}
	- k e^{-\gamma \tau_2} e^{-\lambda(\tau_1+\tau_2)}.
	\]
	Equivalently,
	\begin{equation}
		\lambda^{\alpha} + \gamma
		- k e^{-\lambda \tau_1}
		+ k e^{-\gamma \tau_2} e^{-\lambda(\tau_1+\tau_2)}=0.
		\label{char}
	\end{equation}
	
	\medskip

	Define
	\[
	\Delta(\lambda)
	:= \lambda^{\alpha} + \gamma
	- k e^{-\lambda \tau_1}
	+ k e^{-\gamma \tau_2} e^{-\lambda(\tau_1+\tau_2)}.
	\]
	Evaluating at $\lambda=0$, we get
	\[
	\Delta(0)
	= \gamma - k + k e^{-\gamma \tau_2}.
	\]
	
	The sign of $\Delta(0)$ plays a key role in determining the existence of positive real roots.
	
	\medskip

	For real $\lambda>0$, note that
	\[
	\lambda^{\alpha} \to +\infty \quad \text{as } \lambda \to +\infty,
	\]
	while the exponential terms remain bounded. Hence,
	\[
	\Delta(\lambda)\to +\infty \quad \text{as } \lambda\to+\infty.
	\]
	
	Thus, if $\Delta(0)<0$, then by continuity there exists $\lambda>0$ such that $\Delta(\lambda)=0$. This yields a real positive root, implying instability.
	
	\medskip
	\noindent

	We have
	\[
	\Delta(0)<0
	\quad \Longleftrightarrow \quad
	\gamma - k + k e^{-\gamma \tau_2} < 0,
	\]
	which simplifies to
	\[
	e^{-\gamma \tau_2} < \frac{k-\gamma}{k}.
	\]
	
	\medskip
	\noindent
	\textbf{Case (i): $0<\gamma<k$.}\\

	Here $\frac{k-\gamma}{k}\in(0,1)$. Taking logarithm, we get,
	\[
	-\gamma \tau_2 < \log\left(\frac{k-\gamma}{k}\right).
	\]
	
	\[
\Leftrightarrow	\tau_2 > -\frac{1}{\gamma}\log\left(\frac{k-\gamma}{k}\right).
	\]
	Thus, $\Delta(0)<0$, and hence a positive real root exists. Therefore, the equilibrium is unstable. The condition is independent of $\tau_1$, so instability holds for all $\tau_1\ge 0$.
	
	\medskip
	\noindent
	\textbf{Case (ii): $\gamma<0<k$.}
\hspace{0.2cm}	
	Now $\frac{k-\gamma}{k}>1$. Again,
	\[
	e^{-\gamma \tau_2} < \frac{k-\gamma}{k}.
	\]
	Taking logarithm, we get,
	\[
	-\gamma \tau_2 < \log\left(\frac{k-\gamma}{k}\right).
	\]

	\[
\Leftrightarrow	\tau_2 < -\frac{1}{\gamma}\log\left(\frac{k-\gamma}{k}\right).
	\]
	Thus, $\Delta(0)<0$, implying existence of a positive real root and hence instability for all $\tau_1\ge 0$.
	
	\medskip
	\noindent

	The condition for instability depends only on $\Delta(0)$, which does not involve $\tau_1$. Therefore, the instability is delay-independent with respect to $\tau_1$.
	
	\medskip
	\noindent
	Hence, in both cases, the equilibrium $x_*=0$ is unstable for all $\tau_1\ge 0$ whenever the stated condition on $\tau_2$ holds.
\end{proof}
\textbf{Example 5.1:} $\alpha = 0.3,\; k = 1.4,\; \gamma = 0.8$

Since $0 < \gamma < k$, the system is unstable for
\[
\tau_2 > -\frac{1}{\gamma} \log\left(\frac{k - \gamma}{k}\right) = 1.0591,
\quad \forall \; \tau_1 \ge 0.
\]
Let $\tau_2 = 1.2 > 1.0591$. Then the system is unstable. 
Choose $\tau_1 = 2.3$ and $\tau_1 = 4.7$ (see Fig.~\ref{fig:ex1}).

\begin{figure}[h]
	\centering
	\begin{subfigure}{0.3\textwidth}
		\centering
		\includegraphics[height=4.5cm,  clip]{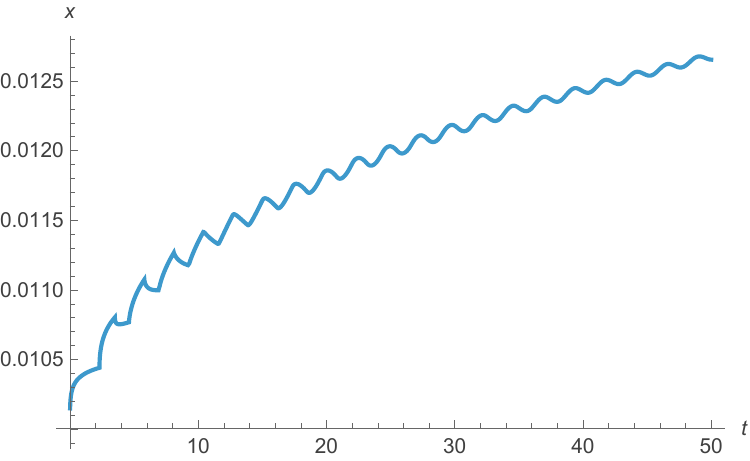}
		\caption{$\tau_1 = 2.3$}
	\end{subfigure}
\hspace{0.2\textwidth}
	\begin{subfigure}{0.3\textwidth}
		\centering
		\includegraphics[height=4.5cm, clip]{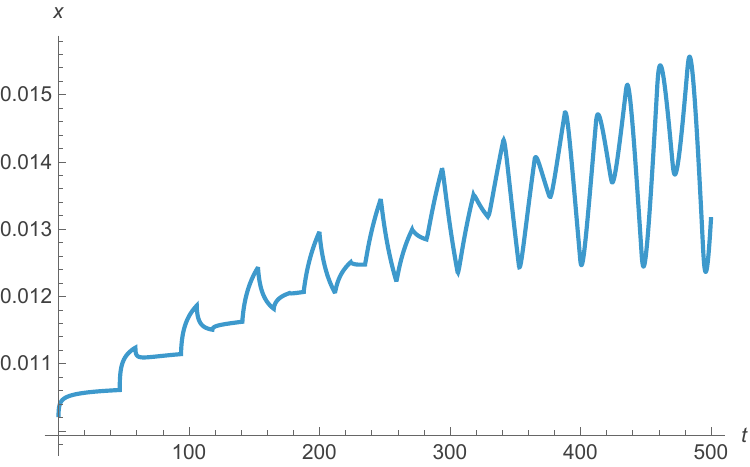}
		\caption{$\tau_1 = 4.7$}
	\end{subfigure}
	\caption{Unstable behavior for Example 1}
	\label{fig:ex1}
\end{figure}

\textbf{Example 5.2:} $\alpha = 0.8,\; k = 3.4,\; \gamma = -1.6$

Since $\gamma < 0 < k$, the system is unstable for
\[
\tau_2 < -\frac{1}{\gamma} \log\left(\frac{k - \gamma}{k}\right) = 0.2410,
\quad \forall \; \tau_1 \ge 0.
\]
Let $\tau_2 = 0.15 < 0.2410$. Then the system is unstable. 
Choose $\tau_1 = 3.4$ and $\tau_1 = 6.1$ (see Fig.~\ref{fig:ex2}).

\begin{figure}[H]
	\centering
	\begin{subfigure}{0.3\textwidth}
		\centering
		\includegraphics[height=4.5cm,  clip]{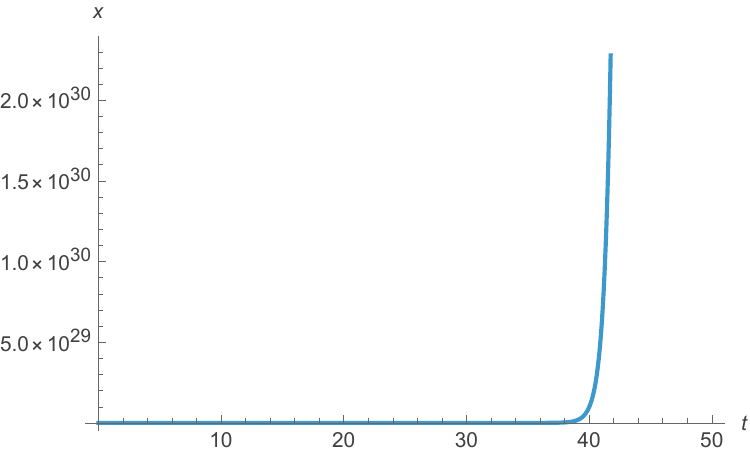}
		\caption{$\tau_1 = 3.4$}
	\end{subfigure}
\hspace{0.2\textwidth}
	\begin{subfigure}{0.3\textwidth}
		\centering
		\includegraphics[height=4.5cm, clip]{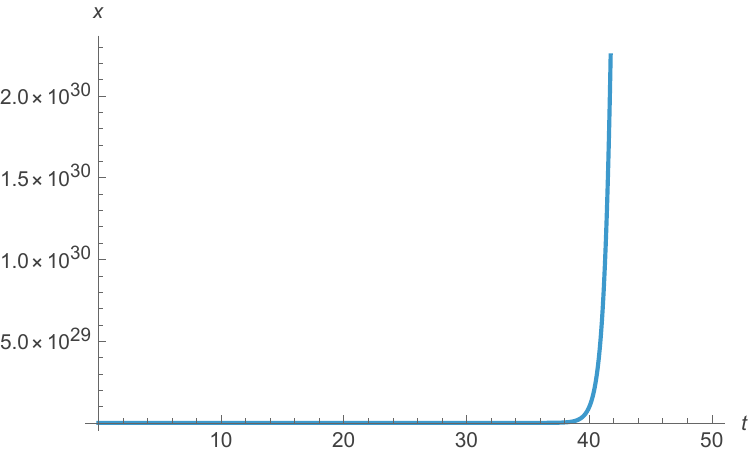}
		\caption{$\tau_1 = 6.1$}
	\end{subfigure}
	\caption{Unstable behavior for Example 2}
	\label{fig:ex2}
\end{figure}
	\section{Stability Diagram in the $\tau_1$–$\tau_2$ Plane}
	\label{sec:C3}
	In this section, we present examples illustrating the stability diagrams in the $\tau_1$–$\tau_2$ plane. 
	To obtain these diagrams, we solve the system obtained by separating the real and imaginary parts of equation~(\ref{eq:char}) after substituting $\lambda = i v$:
	
	\begin{align}
		\begin{split}
			v^{\alpha} \cos\!\left(\dfrac{\alpha \pi}{2}\right)
			+ \gamma 
			- k \cos(v \tau_1)
			+ k e^{-\gamma \tau_2} \cos\!\big(v(\tau_1 + \tau_2)\big)
			= 0, \\[1em]
			v^{\alpha} \sin\!\left(\dfrac{\alpha \pi}{2}\right)
			+ k \sin(v \tau_1)
			- k e^{-\gamma \tau_2} \sin\!\big(v(\tau_1 + \tau_2)\big)
			= 0.
		\end{split}
		\label{eqn:system}
	\end{align}
	
	These equations are solved for positive pairs $(\tau_1, \tau_2)$ by taking $v \in (0, 2\pi)$ by using the "FindRoot" command in Mathematica. 
	For certain values of $v$, the corresponding $\tau_1$ may become negative. 
	However, by adding $2\pi/v$ to those values, we obtain positive ones. 
	This correction is valid because of the periodic nature of the trigonometric terms in~(\ref{eqn:system}). 
	All such positive pairs $(\tau_1, \tau_2)$ are then plotted to form the stability diagram in the $\tau_1$–$\tau_2$ plane. Below are the examples in which the mentioned procedure is being followed, and stability diagrams in the $\tau_1-\tau_2$ plane are presented.
	
\textbf{Example 6.1}
	
	Consider the parameters $\alpha = 0.4$, $k = 1.02$, and $\gamma = 0.3$. 
	The resulting plot is shown in Fig. (\ref{eg5_1}).
	
	\begin{figure}[h!]
		\centering
		\includegraphics[scale=0.4]{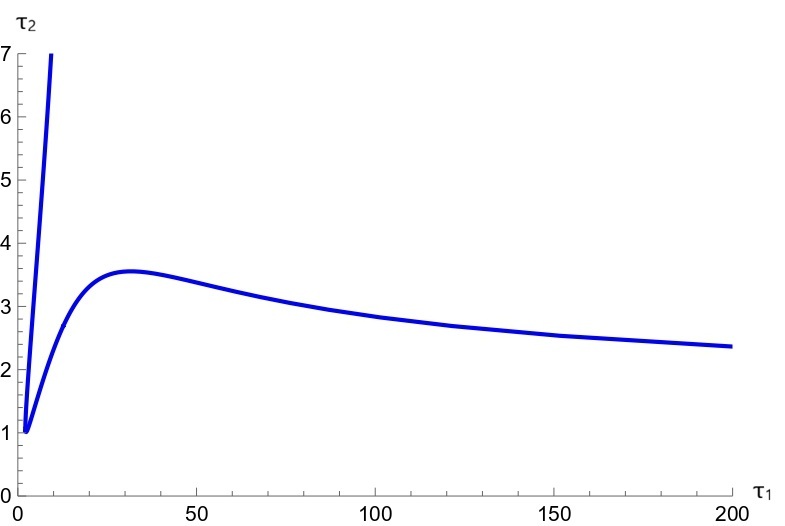}
		\caption{Stability diagram in the $\tau_1$–$\tau_2$ plane for $\alpha=0.4$, $k=1.02$, and $\gamma=0.3$.}
		\label{eg5_1}
	\end{figure}
	
	Ideally, the stable region should extend from the origin up to the boundary of the diagram. 
	However, in this case, that does not occur. 
	The boundary corresponds to $\operatorname{Re}(\lambda) = 0$, but there exists a critical value of $\tau_2 > 0$ where $\lambda = 0$. 
	Substituting $\lambda = 0$ into the characteristic equation~(\ref{eq:char}) gives
	\[
	\gamma - k + k e^{-\gamma \tau_{2a*}} = 0,
	\]
	which simplifies to
	\begin{equation}
		\label{tau2}
		\tau_{2a*} = \frac{-1}{\gamma}\log\!\left(\frac{k - \gamma}{k}\right).
	\end{equation}
	For the given parameter values, we find $\tau_{2a*} = 1.16102$.
	
	\begin{figure}[h!]
		\centering
		\includegraphics[width=0.5\linewidth]{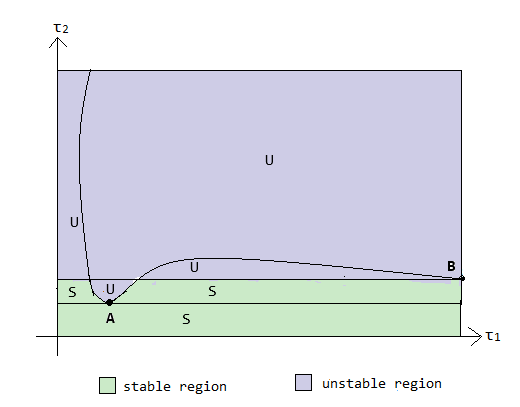}
		\caption{Critical points on the stability boundary in the $\tau_1$–$\tau_2$ plane (Not upto scale).}
		\label{eg5_2}
	\end{figure}
	
	This point is denoted as point~B in Fig. \ref{eg5_2}. 
	At the minimum point A, $ \tau_2=\tau_{2b*} = 1.0674$, which is less than $\tau_{2a*}$. 
	Hence, $\tau_{2b*}$ is the first critical value of $\tau_2$. 
The stability behavior can be summarized as follows:
\begin{itemize}
	\item For $0 < \tau_2 < \tau_{2b*}$, the equilibrium remains \textbf{stable} for all values of $\tau_1$.
	\item For $\tau_{2b*} < \tau_2 < \tau_{2a*}$, a switch \textbf{Stable-Unstable-Stable (SUS)} exists in terms of $\tau_1$, as illustrated in Fig.~\ref{eg5_2}.
	\item For $\tau_2 > \tau_{2a*}$, the equilibrium becomes \textbf{unstable} for all values of $\tau_1$.
\end{itemize}

Examples to verify these behaviors:
\begin{enumerate}
	\item Let $\tau_2 = 1.04 < \tau_{2b*}$. In this case, the equilibrium is stable for all $\tau_1$. For instance, choosing $\tau_1 = 0.2$, $\tau_1 = 1.6$, and $\tau_1 = 2.4$ confirms the stability (see Fig.~\ref{3_6_ex1}).
	
	\begin{figure}[H]
		\centering
		\begin{subfigure}{0.3\textwidth}
			\centering
			\includegraphics[width=\linewidth]{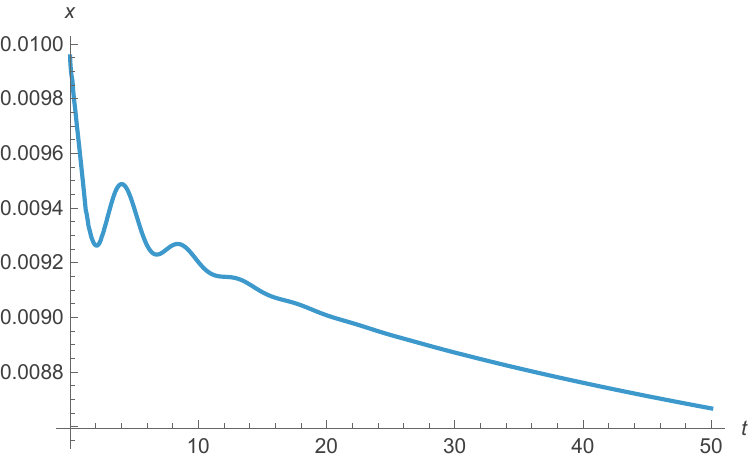}
			\caption{$\tau_1 = 0.2$}
		\end{subfigure}
		\hfill
		\begin{subfigure}{0.3\textwidth}
			\centering
			\includegraphics[width=\linewidth]{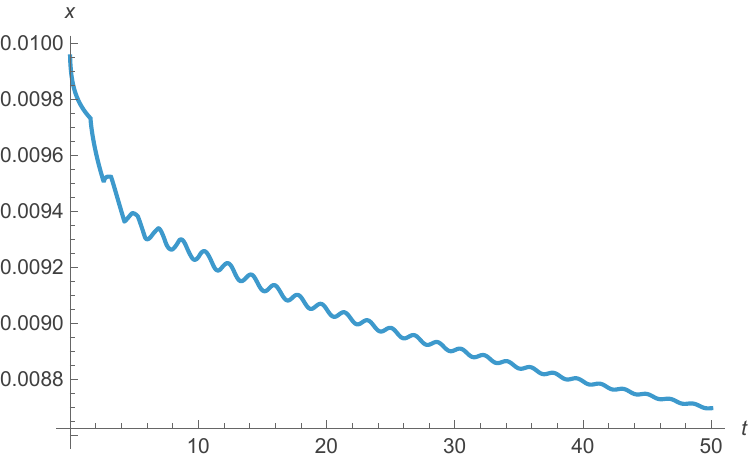}
			\caption{$\tau_1 = 1.6$}
		\end{subfigure}
		\hfill
		\begin{subfigure}{0.3\textwidth}
			\centering
			\includegraphics[width=\linewidth]{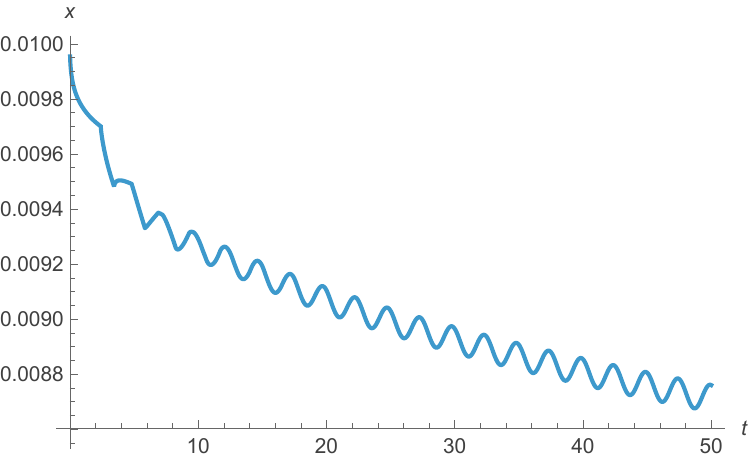}
			\caption{$\tau_1 = 2.4$}
		\end{subfigure}
		\caption{Stable behavior for $\tau_2 = 1.04$}
		\label{3_6_ex1}
	\end{figure}
	
\item Let $\tau_2 = 1.1$, so that $\tau_{2b*} < \tau_2 < \tau_{2a*}$. In this regime, the equilibrium exhibits an S-U-S type stability switching (see Fig.~\ref{3_6_ex2}). For $\tau_2 = 1.1$, the boundary values are $\tau_{1a} = 2.11$ and $\tau_{1b} = 3.03$.

Choosing $\tau_1 = 1.8 < \tau_{1a}$, $\tau_{1a} < \tau_1 = 2.9 < \tau_{1b}$, and $\tau_1 = 3.5 > \tau_{1b}$, we observe that the system is stable for $\tau_1 = 1.8$ and $\tau_1 = 3.5$. However, for $\tau_1 = 2.9$, a characteristic root $\lambda = 0.00200287 + 2.10566i$ with positive real part appears, indicating instability.

\begin{figure}[H]
	\centering
	\begin{subfigure}{0.45\textwidth}
		\centering
		\includegraphics[width=\linewidth]{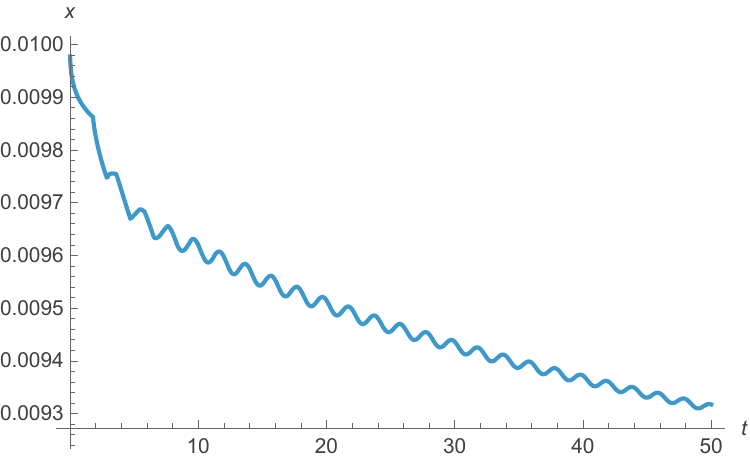}
		\caption{$\tau_1 = 1.8$ (stable)}
	\end{subfigure}
	\hfill
	\begin{subfigure}{0.45\textwidth}
		\centering
		\includegraphics[width=\linewidth]{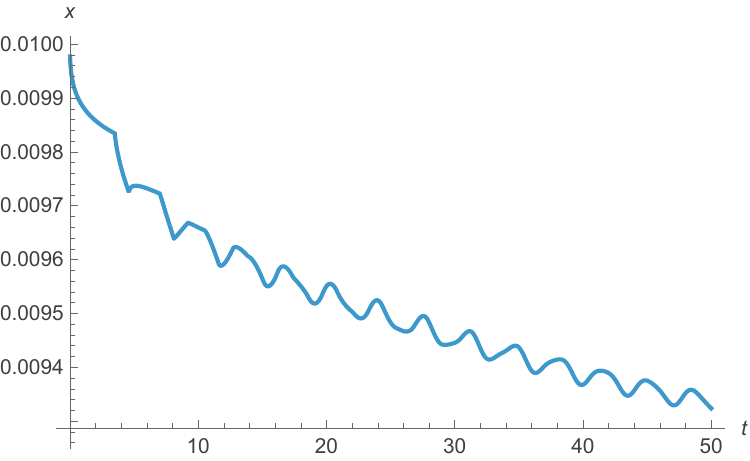}
		\caption{$\tau_1 = 3.5$ (stable)}
	\end{subfigure}
	\caption{Stable cases corresponding to the S-U-S stability switching for $\tau_2 = 1.1$; instability occurs for intermediate values $\tau_{1a} < \tau_1 < \tau_{1b}$.}
	\label{3_6_ex2}
\end{figure}
		\item Let $\tau_2 = 1.21 > \tau_{2a*}$. In this case, the equilibrium is unstable for all $\tau_1$. For instance, choosing $\tau_1 = 0.3$, $\tau_1 = 1.8$, and $\tau_1 = 3.2$ confirms the instability (see Fig.~\ref{3_6_ex3}).
			\begin{figure}[H]
				\centering
				\begin{subfigure}{0.23\textwidth}
					\centering
					\includegraphics[height=4.5cm, clip]{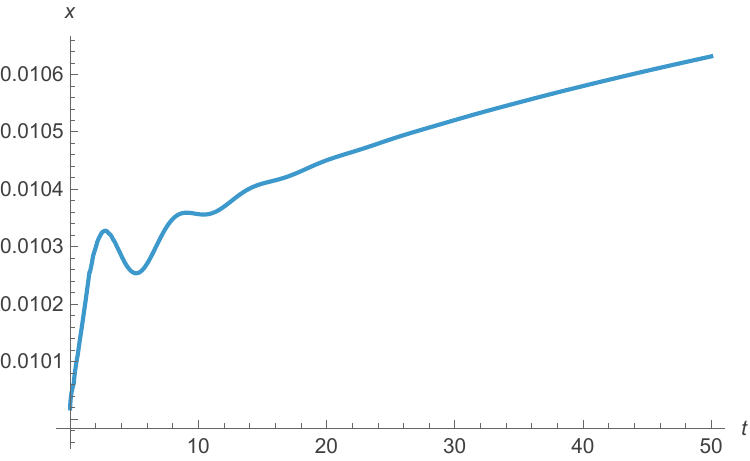}
					\caption{$\tau_1 = 0.3$}
				\end{subfigure}
			\hspace{0.1\textwidth}
				\begin{subfigure}{0.23\textwidth}
					\centering
					\includegraphics[height=4.5cm, clip]{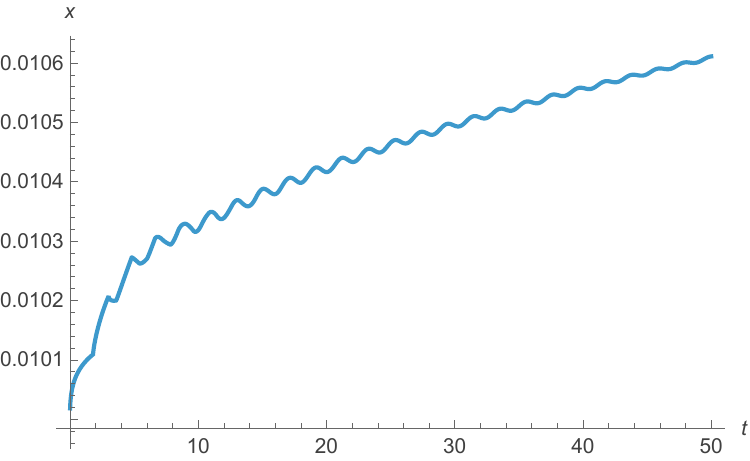}
					\caption{$\tau_1 = 1.8$}
				\end{subfigure}
			\hspace{0.1\textwidth}
				\begin{subfigure}{0.23\textwidth}
					\centering
					\includegraphics[height=4.5cm, clip]{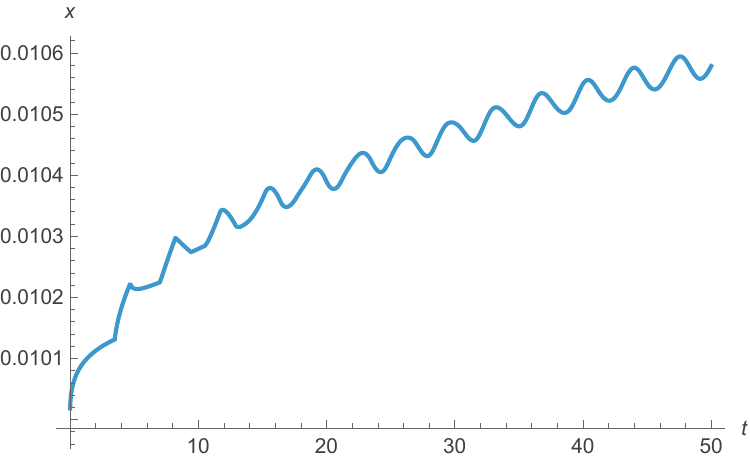}
					\caption{$\tau_1 = 3.5$}
				\end{subfigure}
				\caption{Stability switching behavior for $\tau_2 = 1.1$}
				\label{3_6_ex3}
			\end{figure}
	
\end{enumerate}
	\textbf{Note:} In Example 1, two branches arise, namely $\lambda = iv$ (first branch) and $\lambda = 0$ (second branch). We seek a condition under which the second branch becomes irrelevant. Observe that if the critical value $\tau_{2*}$ corresponding to $\lambda = 0$ is negative, then this branch does not influence the stability. Therefore, we determine the condition under which the expression for $\tau_{2*}$ given in (\ref{tau2}) is negative. For this purpose, we consider the following two cases:\\
\textbf{Case 1: $\gamma>0$}\\
For 
\[
\frac{-1}{\gamma}\log\!\left(\frac{k - \gamma}{k}\right) < 0,
\]
we must have
\begin{align}
	\log\!\left(\frac{k - \gamma}{k}\right) &> 0,\nonumber\\[0.5em]
	\text{i.e.,} \quad 1 - \frac{\gamma}{k} &> 1, \nonumber\\[0.5em]
	\implies -\frac{\gamma}{k} &> 0, \nonumber\\[0.5em]
	\implies k &< 0. \nonumber
\end{align}
\textbf{Case 2: $\gamma<0$}\\

For 
\[
\frac{-1}{\gamma}\log\!\left(\frac{k - \gamma}{k}\right) < 0,
\]
we must have
\begin{align}
	\log\!\left(\frac{k - \gamma}{k}\right) &< 0, \nonumber\\[0.5em]
	\text{i.e.,} \quad 0 < 1 - \frac{\gamma}{k} &< 1. \nonumber
\end{align}

This yields
\begin{align}
	1 - \frac{\gamma}{k} &< 1 \;\;\implies\;\; \frac{\gamma}{k} > 0, \nonumber\\[0.5em]
	0 < 1 - \frac{\gamma}{k} &\;\;\implies\;\; \frac{\gamma}{k} < 1. \nonumber
\end{align}

Hence,
\begin{align}
	0 < \frac{\gamma}{k} < 1. \nonumber
\end{align}

Now, we consider two subcases:

\begin{itemize}
	\item If $k > 0$, then $0 < \frac{\gamma}{k} < 1$ implies $0 < \gamma < k$, which contradicts the assumption $\gamma < 0$.
	
	\item Therefore, we must have $k < 0$. In this case, the inequality $0 < \frac{\gamma}{k} < 1$ implies
	\[
	k < \gamma < 0.
	\]
\end{itemize}
	
	Both of these cases imply that in both the second and the third quadrant, the critical value obtained by putting $\lambda=0$ in the characteristic equation (\ref{eq:char}) will not affect the stability.
	
	So, now we consider the same example with negative $k$, i.e, $(k,\gamma)$ lies in the second quadrant.
	\textbf{Example 6.2}
	
	Let us consider the parameters $\alpha = 0.4$, $k =- 1.02$, and $\gamma = 0.3$. 
	The resulting plot is shown in Fig.~\ref{eg5_3}.
	
	\begin{figure}[h!]
		\centering
		\includegraphics[width=0.5\linewidth]{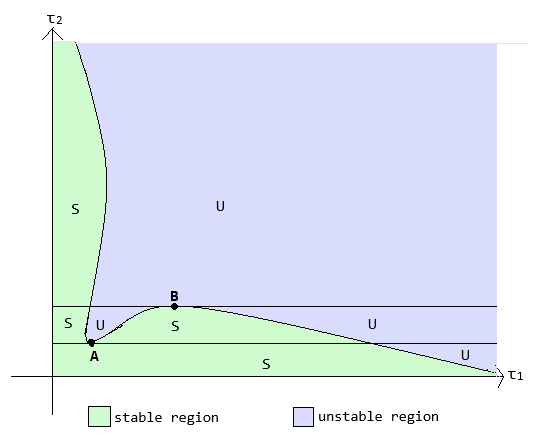}
		\caption{Stability diagram in the $\tau_1$–$\tau_2$ plane for $\alpha=0.4$, $k=-1.02$, and $\gamma=0.3$(Not upto scale).}
		\label{eg5_3}
	\end{figure}
	Here, $\lambda=0$ gives $\tau_{2*}=-0.85943$, which cannot be a critical value. So, the stability boundary will be given by $\lambda=iv.$ Hence, in this example, we can observe
	both the SSR region and the S-U-S-U region, i.e, say the local minima point $A$ and the local maxima point correspond to the critical values $\tau_{2a*}$ and $\tau_{2b*}$ so we have the behavior as shown in Fig. \ref{eg5_3}.
	
		\section{Conclusion}
	\label{sec:conc}

In this paper, we have extended our earlier work \cite{dutta2025some} on nonlinear fractional delay differential equations with two discrete delays by completing the stability analysis in the remaining regions of the $(k,\gamma)$-plane. In particular, for the case $\tau_1 = 0$, a full characterization of the stability behavior has been obtained, thereby filling the gaps left in the previous study.
For the general case involving two delays, we derived a new stability result that provides a sufficient condition for instability in certain parameter regions. 

In addition, we investigated the combined influence of the two delays by constructing stability diagrams in the $(\tau_1,\tau_2)$-plane for selected parameter values. 
Overall, the results presented in this work contribute to a more complete understanding of fractional delay systems with multiple delays and delay-dependent coefficients. Future work may consider more general nonlinearities, distributed delays, and higher-dimensional systems.

	\bibliographystyle{unsrt}	
\bibliography{reff}
	
	\end{document}